\documentclass[12pt,reqno]{amsart}   	
\usepackage[letterpaper, margin=1in]{geometry}
\usepackage{graphicx}				
\usepackage{amssymb}
\usepackage{mathtools,amsmath}
\usepackage{amsthm}
\usepackage{amssymb}
\usepackage{mathrsfs}
\usepackage{epstopdf}
\usepackage{color}
\usepackage{enumerate}
\usepackage[pagebackref,colorlinks,citecolor=blue,linkcolor=magenta]{hyperref}
\usepackage[capitalise]{cleveref}

\usepackage[linesnumbered,ruled]{algorithm2e}
\RequirePackage{amsthm,amsmath,amsfonts,amssymb}
\usepackage[utf8]{inputenc}
\usepackage{chngcntr}
\usepackage{float}

\theoremstyle{plain}
\newtheorem{prop}{Proposition}[section]
\newtheorem{thm}[prop]{Theorem}
\newtheorem{cor}[prop]{Corollary}
\newtheorem{lem}[prop]{Lemma}

\theoremstyle{definition}
\newtheorem{dfn}[prop]{Definition}
\newtheorem{rem}[prop]{Remark}
\newtheorem{example}[prop]{Example}

\renewcommand{\iff}{\Leftrightarrow}

\newcommand{\C}{{\mathbb{C}}}
\renewcommand{\P}{{\mathbb{P}}}
\newcommand{\R}{{\mathbb{R}}}
\newcommand{\Q}{{\mathbb{Q}}}
\newcommand{\N}{{\mathbb{N}}}
\newcommand{\Z}{{\mathbb{Z}}}

\DeclareMathOperator{\im}{im}
\DeclareMathOperator{\tr}{tr}
\DeclareMathOperator{\interior}{int}
\DeclareMathOperator{\rk}{rk}
\DeclareMathOperator{\crk}{\dim\ker}
\DeclareMathOperator{\diag}{diag}
\DeclareMathOperator{\gl}{GL}
\DeclareMathOperator{\Sym}{Sym}
\DeclareMathOperator{\gram}{Gram}
\DeclareMathOperator{\aff}{aff}
\DeclareMathOperator{\pr}{pr}
\DeclareMathOperator{\spn}{span}

\DeclareMathOperator{\normal}{\mathcal{N}}
\DeclareMathOperator{\dphi}{d\phi}
\DeclareMathOperator{\disc}{disc}
\newcommand{\wS}{\widetilde{S}}
\renewcommand{\d}[1]{\,\mathrm{d}#1}
\DeclareMathOperator{\adj}{adj}

\newcommand{\Rx}{\mathbb{R}[\ul{x}]}

\newcommand{\oo}{\mathcal{O}}
\newcommand{\uu}{\mathcal{U}}

\newcommand{\du}{{\scriptscriptstyle\vee}}
\renewcommand{\setminus}{\smallsetminus}
\newcommand{\ol}{\overline}
\newcommand{\ul}{\underline}

\newcommand{\all}{\forall\,}

\renewcommand{\subset}{\subseteq}
\renewcommand{\supset}{\supseteq}
\newcommand{\bil}[2]{\langle{#1},{#2}\rangle}

\newcommand{\sy}[1]{\mathcal{S}_2 #1}

\newcommand{\sa}{semi-alge\-braic}

\newcommand{\gs}{Gram spectrahedron}
\newcommand{\gsa}{Gram spectrahedra}
\newcommand{\cc}{change of coordinates}
\newcommand{\norm}[1]{||#1||}

\newcommand{\minus}{\text{-}}

\renewcommand{\emptyset}{\varnothing}
\renewcommand{\setminus}{\smallsetminus}
\renewcommand{\epsilon}{\varepsilon}
\renewcommand{\theta}{\vartheta}

\newcommand{\todfn}[1]{\textit{#1}}

\newcommand\scalemath[2]{\scalebox{#1}{\mbox{\ensuremath{\displaystyle #2}}}}


\author[Julian Vill]{Julian Vill}
\address{
Fakult\"at f\"ur Mathematik, Otto-von-Guericke Universit\"at Magdeburg, Universit\"atsplatz 2, 39106 Magdeburg, Germany
}
\email{julian.vill@ovgu.de}

\title{Integrating Spectrahedra}

\begin{document}

\begin{abstract}
Given a linear map on the vector space of symmetric matrices, every fiber intersected with the set of positive semidefinite matrices is a spectrahedron. Using the notion of the fiber body we can build the average over all such fibers and thereby construct a compact, convex set, the fiber body. We show how to determine the dimensions of faces and normal cones to study the boundary structure of the fiber body given that we know about the boundary of the fibers.
We use this to study the fiber body of Gram spectrahedra in the case of binary sextics and ternary quartics and find a large amount of structure on the fiber body. We prove that the fiber body in the case of binary sextics has exactly one face with a full-dimensional normal cone, whereas the Gram spectrahedron of a generic positive binary sextic has four such points. The fiber body in the case of ternary quartics changes drastically.
\end{abstract}

\maketitle

\section{Introduction}

In \cite{MM21} Mathis and Meroni study the fiber body of a compact convex set. Informally, this means the following. Choose a convex, compact set $K\subset\R^{n+m}$ and consider the projection $\pi\colon \R^{n+m}\to \R^n$ to the first $n$ components. Over each point in the image we have a convex fiber in $K$. The fiber body of $K$ is the average over all fibers.
More precisely, any point $y$ in the fiber body is given by
\[
y=\int_{\pi(K)} \gamma(x) \d x
\]
where $\gamma\colon \pi(K)\to \R^m$ is a measurable section. The fiber body $\Sigma_\pi K$ is itself a compact convex set in $\R^m$. It is the continuous analogue of the Minkowski sum of convex sets. One may also think about a limit object of Minkowski sums when the number of summands tends to infinity.
In \cite{MM21} the fiber body was then studied in several special cases where the support function of the convex body $K$ is known.

Before that the notion of fiber polytopes was introduced and studied in \cite{bs1992}. These possess beautiful and well-studied combinatorics which led to numerous constructions of polytopes with prescribed combinatorial structure, for example in \cite{z1995}.

In this paper we study the fiber body of families of spectrahedra. We use a different approach as we do not in general know the support function. The goal is to understand the facial structure of the fiber body given we know the facial structure of the fibers.
We then apply the general setup to study two special families of spectrahedra, namely \gsa\ of binary sextics and of ternary quartics. These are examples in dimensions 3 and 6 respectively.

Gram spectrahedra have a close connection to sum of squares representations of polynomials and have been introduced in \cite{clr1995}. The \gs\ of a fixed homogeneous polynomial parametrizes its sum of squares representations up to orthogonal equivalence. Afterward, especially the structure of their extreme points was investigated, for example in \cite{cpsv2017} and \cite{psv2011}. In \cite{Scheiderer18} and \cite{Vill21} the complete facial structure of Gram spectrahedra was studied in the case of ternary quartics and binary forms.
As we do not have a description of the support function of Gram spectrahedra, we make extensive use of these results in order to describe the facial structure of the fiber body.

In both cases we were very surprised by the remaining facial structure on the fiber bodies. In the case of binary sextics we show that there is exactly one extreme point on the boundary of the fiber body with a 3-dimensional normal cone (\cref{thm:binary_sextics_fb}). A general positive semidefinite (psd) binary sextic has exactly four such points corresponding to length two sum of squares representations.
In the case of ternary quartics the structure of the boundary changes rather drastically. On the one hand, every face has a 1-dimensional normal cone even though the \gs\ of a general psd ternary quartic has normal cones of dimensions $1,3,6$. On the other hand, the fiber body has a 2-dimensional family of 3-dimensional faces whereas a general \gs\ has no faces of dimension larger than 2 (\cref{thm:ternary_quartics_fb}).

As a byproduct (\cref{prop:algebraic_degree}) we are able to explain the algebraic degree of the optimal solution when minimizing a linear functional over the \gs\ of a ternary quartic found in \cite{psv2011} when the solution is a rank 5 extreme point. This algebraic degree was found to be 1 and therefore the solution being rational if both the ternary quartic and the linear functional are rational. This follows from a detailed study of the normal cones of such extreme points.

We finish with an overview of the structure of the paper. In \cref{sec:intro_fb} we define the fiber body and recall some general facts we need later on. In \cref{sec:fb_spec} we introduce our setup for the rest of the paper. We study the fiber body of spectrahedra and show \cref{thm:dimension_face_fb} and \cref{thm:normal_cone_fb} which are our main tools to relate the faces and normal cones of the fibers to the one of the fiber body.
In \cref{sec:faces_of_gram_spectrahedra} we specialize the setup to \gsa\ and recall what is known about the dimensions of their faces. Moreover, we give a formula to calculate the dimensions of the normal cones.
In \cref{sec:binary_sextics} and \cref{sec:ternary_quartics} we study the fiber body in the case of \gsa\ of binary sextics and ternary quartics. The main results are \cref{thm:binary_sextics_fb} and \cref{thm:ternary_quartics_fb} which summarize the facial structure of the fiber body in each case.

\section{Introduction to fiber bodies}
\label{sec:intro_fb}

We start by introducing the fiber body construction from \cite{MM21}.
Let $V\subset\R^{n+m}$ ($m,n\in\N$) be a subspace of dimension $n$ and let $W$ be its orthogonal complement. Let $\pi\colon \R^{n+m}\to V$ be the orthogonal projection to $V$. For every point $v\in V$ we consider the fiber $\pi^{-1}(v)$ which we identify with a subset of $W$.

Let $K\subset \R^{n+m}$ be a compact, convex set. For any $v\in\pi(K)$ we write $K_v=\pi^{-1}(v)\cap K$ for the fiber of $v$ intersected with $K$.

\begin{dfn}
A map $\gamma\colon \pi(K)\to W,\, x\to \gamma(x)$ with $\gamma(x)\in K_x$ for all $x\in \pi(K)$ is called a \todfn{section (of $\pi$)}.

The \todfn{fiber body} of $K$ with respect to $\pi$ is the compact, convex set
\[
\Sigma_\pi K=\left\lbrace \int_{\pi(K)} \gamma(x) \d x\colon \gamma\colon \pi(K)\to W \text{ a measurable section}\right\rbrace
\]
where $\d x$ denotes integration with respect to the $n$-dimensional Lebesgue measure on $V$.

The \todfn{support function} of $K$ is the map
\[
h_K\colon \R^{n+m}\to \R,\quad u\mapsto h_K(u)=\min\{\bil{u}{x}\colon x\in K\}\quad (u\in\R^{n+m})
\]
where $\bil{\cdot}{\cdot}$ is the standard scalar product. We use minimum in the definition instead of maximum since it works better with dual cones we use in later sections.
\end{dfn}

\begin{prop}[{\cite[Proposition 2.7.]{MM21}}]
For any $w\in W$ the support function of the fiber body satisfies
\[
h_{\Sigma_\pi K}(w)=\int_{\pi(K)} h_{K_x}(w) \d x.
\]
\end{prop}

For any $w\in W$ we denote by $K^w$ the face of $K$ in direction $w$ which is 
\[
K^w=\{x\in K\colon \bil{x}{w}\le \bil{y}{w}\, (\all y\in K)\}.
\]
This is what is usually called an exposed face. Since we are concerned with spectrahedra, most of the faces we are interested in are exposed. In general, a convex subset $F\subset K$ is a face of $K$ if for any two points $x,y\in K$ with $x+y\in F$, both points $x,y$ are contained in $F$.

The support function of faces behaves just as well.

\begin{lem}[{\cite[Lemma 2.10.]{MM21}}]
\label{lem:support_function_faces_fb}
Let $u,w\in W$ then
\[
h_{(\Sigma_\pi K)^u}(w)=\int_{\pi(K)} h_{(K_x)^u}(w) \d x.
\]
\end{lem}

As one might expect, extreme points are particularly well-behaved.

\begin{prop}[{\cite[Lemma 2.9., Proposition 2.14.]{MM21}}]
\label{prop:fb_extreme_points}
Let $w\in W$. Then $(\Sigma_\pi K)^w$ is an extreme point if and only if $(K_x)^w$ is an extreme point for almost all $x\in\pi(K)$.

Moreover, if $(\Sigma_\pi K)^w=\{z\}$ is an extreme point and $(K_x)^w=\{z_x\}$ for almost all $x\in\pi(K)$, then $z=\int_{\pi(K)} \gamma(x) \d x$ with $\gamma(x)=z_x$ for almost all $x\in\pi(K)$.
Especially, the section $\gamma$ is unique almost everywhere.
\end{prop}

\begin{example}
Consider the unit disk $K\subset \R^2$ projected onto the $x$-axis. The image of $K$ is the unit interval and the fiber over almost every points is an interval. The fiber body is then also an interval whose boundary points are $\int_{\pi(K)} \gamma(x) \d x$ with $\gamma(x)$ being the upper boundary point of every fiber, i.e. the face in direction $-y$, respectively the face in direction $y$.
\end{example}

More generally, the following Lemma holds.

\begin{lem}[{\cite[Lemma 2.9.]{MM21}}]
\label{lem:faces_fb}
Let $\uu=\{w_1,\dots,w_k\}$ be an ordered family of linearly independent vectors in $W$. Then 
\[
(\Sigma_\pi K)^\uu=\left\lbrace \int_{\pi(K)} \gamma(x) \d x\colon \gamma \text{ section with } \gamma(x)\in (K_x)^\uu \text{ for all } x\in\pi(K) \right\rbrace
\]
where $(\Sigma_\pi K)^\uu$ and $(K_x)^\uu$ denote iterated faces, i.e. $(K_x)^\uu=(\dots((K_x)^{w_1})^{w_2}\dots)^{w_k}$.
\end{lem}

\section{Fiber bodies of spectrahedra}
\label{sec:fb_spec}

As our focus is on fiber bodies of spectrahedra we now specialize to this setup. We recall some general facts about the cone of psd matrices and show \cref{thm:dimension_face_fb} and \cref{thm:normal_cone_fb} which describe the faces and their normal cones of the fiber body. These we use in the following sections as our main tools to study the fiber body.

Let $V$ be a finite-dimensional $\R$-vector space. We write $\Sym^N$ for the vector space of real symmetric matrices of size $N\times N$ and $\Sym^N_+$ for the subset of positive semidefinite matrices. Consider a linear map $\pi\colon\Sym^N\to V$. We assume that $\pi$ is surjective as otherwise the fiber body is $\{0\}$. For any $v\in V$ the fiber $\pi^{-1}(v)\cap\Sym^N_+$ is a spectrahedron.

In order for the fiber body construction to work properly we require $\ker(\pi)\cap\Sym_+^N=0$. This will ensure compact fibers.
We denote by $W$ the kernel of the map $\pi$. Choose any linear complement $L$ of $W$ in $\Sym^N$. Then $\pi$ defines an isomorphism between $V$ and $L$ and by slight abuse of notation we can write $\Sym^N=V\oplus W$.

Let $||\cdot ||$ be any norm on $V$ and let $K\subset\Sym^N_+$ be the preimage of the unit ball under $\pi$ intersected with $\Sym^N_+$. Since $\pi$ is linear, the set $K$ is then a compact, convex, semi-algebraic set. 
Indeed, assume $K$ was not bounded, then the recession cone of $K$ is non-trivial and there exists $0\neq G\in\Sym^N$ such that $K+cG\subset K$ for every $c\in\R_{>0}$. Since $\pi(K)+c\pi(G)=\pi(K+cG)\subset \pi(K)$ is compact we see that $\pi(G)=0$. Hence, $G$ cannot be psd and has a negative eigenvalue. Therefore, for any fixed $x\in K$ and $0\ll c$ the element $x+cG$ is not psd and thus not contained in $K$, a contradiction.
In the same way, every fiber is compact.

We can therefore construct the fiber body $\Sigma_\pi K$ by considering every spectrahedron $K_v$ ($v\in V$) to be a compact subset of $W$. Note that for any $v\in V$ with $\norm{v}\le 1$, $\pi^{-1}(v)\cap\Sym^N_+=\pi^{-1}(v)\cap K$ by choice of $K$ and therefore we did not change these spectrahedra by intersecting with $K$.
Furthermore, even though this is necessary for the fiber body construction, we can mostly ignore the condition $\norm{v}\le 1$ since for any $0\neq v\in V$ we have $\pi^{-1}(v)\cap \Sym^N_+=\norm{v} (\pi^{-1}\left(\frac{v}{\norm{v}}\right)\cap \Sym^N_+)=\norm{v} (\pi^{-1}\left(\frac{v}{\norm{v}}\right)\cap K)$. I.e. all fibers are just scaled version of fibers of elements in with norm at most 1. Especially, they have the same facial structure.

\begin{example}
The easiest but rather uninteresting example is if $\dim V=1$. Consider for example the trace map on $\Sym^N$. For any $\lambda\in\R_{>0}$ the fiber $\tr^{-1}(\lambda)\cap\Sym_+^N$ is a compact cone basis of $\Sym_+^N$ and therefore has basically the same facial structure as $\Sym_+^N$ itself. It is then not hard to see that the same also holds for the fiber body. With the same argument the same holds for any linear functional instead of the trace as long as the fibers are compact.
\end{example}

\begin{lem}
\label{lem:correspondence_functionals}
Let $x=v+w\in V\oplus W$, $v\in V, w\in W$. For every $0\neq u\in V$ the face of $S:=\pi^{-1}(u)\cap\Sym_+^N$ in direction $x$ is the same as the face in direction $w$. In the first case $S$ should be seen as a subset of $\Sym^N$, in the second case either as a subset of $\Sym^N$ or of $W$.
\end{lem}
\begin{proof}
Write $S=G+\wS$ with $G\in V$, $\pi(G)=u$ and $\wS\subset W$. Then for any $H\in\wS$ we have
\[
\bil{x}{G+H}=\bil{v+w}{G+H}=\bil{v}{G}+\bil{w}{H}
\]
since $\bil{v}{H}=\bil{w}{G}=0$. As $G$ only depends on $u$ we see that $\bil{x}{G+H}$ is minimal if and only if $\bil{w}{H}$ is minimal.
\end{proof}

The following is well-known about the facial structure of the psd cone itself.

\begin{prop}
\label{prop:dimension_face_spec}
Let $F\subset\Sym_+^N$ be a face with corresponding subspace $U$. Then for any $u\in \pi(F)$, and any $x\in \Sym^N$ with $F=(\Sym^N_+)^x$, the face of $S_u$ in direction $x$ is given by
\[
\{G\in S_u\colon \im(G)\subset U\}\neq\emptyset.
\]
For any $u\in \pi(\interior F)$ this face has dimension $\binom{\dim U+1}{2}-\dim \pi(F)$.
\end{prop}

The following is one main reason why considering spectrahedra instead of general convex sets makes the study of their fiber bodies easier.

\begin{rem}
\label{rem:completion}
Given $0\neq w\in W$. What is $S_u^w$? One answer is as follows: Assume we find $v\in V$ such that $v+w$ is psd, and consider $\im(v+w)$. Then we know the dimension of $S_u^w$ for any $u\in\pi(\interior(F_{\im(v+w)}))$ by \cref{prop:dimension_face_spec}, where $F_{\im(v+w)}=\{G\in\Sym^N_+\colon \im(G)\subset\im (v+w)\}$.

Assume for example we can find $v\in V$ such that $v+w$ is psd of rank 1 with image $U$. Then $S_u^w$ is a face of corank 1 for any $u\in\pi(\interior(F_U))$.

We use this approach throughout the paper and try to complete elements in $W$ to low rank psd matrices $v+w$.
\end{rem}

Note that the next theorem does not depend on the fact that the fiber body is constructed from spectrahedra. The same proof also works for fiber bodies in general.
For a convex set $C\subset\R^m$ we denote by $\aff(C)$ the affine hull of $C$, i.e. the smallest affine subspace containing $C$.

\begin{thm}
\label{thm:dimension_face_fb}
Let $w\in W$ and let $S_u^w$ be the face in direction $w$ on $S_u$. We denote by $F$ the face $(\Sigma_\pi K)^w$ on $\Sigma_\pi K$.
Then the subspace $A$ corresponding to $\aff(F)$ is the vector space spanned by all subspaces $A_u$ corresponding to the affine hulls of $S_u^w$ in such a way that the dimension is minimal after removing a set of measure zero from $V$, i.e.
\[
A=\bigcap_{\oo'} \spn(A_u\colon u\in\oo')
\]
where the intersection is over all subsets $\oo'\subset V$ such that its complement $(\oo')^c$ has measure zero.
Moreover, there exists $\oo\subset V$ such that $\oo^c$ has measure zero and
\[
A=\spn(A_u\colon u\in\oo).
\]
\end{thm}
\begin{proof}
We write $W_\oo:=\spn(A_u\colon u\in\oo)$. The equality $W_\oo=\bigcap_{\oo'} \spn(A_u\colon u\in\oo')$ follows from the fact that the intersection is a subspace of some dimension. Hence, finitely many $\oo'$ already suffice and we may take $\oo$ to be the intersection of these.
Pick a basis $w,w_2,\dots,w_s$ of the orthogonal complement of $W_\oo$ in $W$. For any $u\in \oo$ we have
\[
S_u^w=S_u^{w,w_2,\dots,w_s}
\]
by choice of $W_\oo$. Thus, by \cref{lem:faces_fb}
\[
F=(\Sigma_\pi K)^{w,w_2,\dots,w_s}
\]
which shows $A\subset W_\oo$.

Assume $\dim A<\dim W_\oo$. Then there exists $w_{s+1}\in W$, $w,w_2,\dots,w_s,w_{s+1}$ linearly independent such that 
\[
F=(\Sigma_\pi K)^{w,w_2,\dots,w_s,w_{s+1}}.
\]
We denote by $h$ the support function of $F$ and by $h_u$ the support function of $S_u^w$. Let $x\in W$ and $\lambda\in\R$. Then by assumption $h(x+\lambda w_{s+1})=h(x)+\lambda h(w_{s+1})$. Hence,
\[
\int h_u(x+\lambda w_{s+1}) \d u\stackrel{\ref{lem:support_function_faces_fb}}{=} h(x+\lambda w_{s+1})=h(x)+\lambda h(w_{s+1})\stackrel{\ref{lem:support_function_faces_fb}}{=}\int h_u(x) \d u +\lambda \int h_u(w_{s+1})\d u.
\]
Moreover, for every $u\in V$ we have
\[
h_u(x+\lambda w_{s+1})\le h_u(x)+\lambda h_u(w_{s+1}), \text{ equivalently } 0\le  h_u(x)+\lambda h_u(w_{s+1})-h_u(x+\lambda w_{s+1}).
\]
Hence,
\[
0= \int  h_u(x)+\lambda h_u(w_{s+1})-h_u(x+\lambda w_{s+1}) \d u
\]
implies $h_u(x+\lambda w_{s+1}) = h_u(x)+\lambda h_u(w_{s+1})$ for almost all $u\in V$.

Being linear in the last coordinate however means that $S_u^w$ is contained in a translate of the orthogonal complement of $\spn(w_{s+1})$. By assumption this does not hold for almost all $u$, hence the assumption $h(x+\lambda w_{s+1})=h(x)+\lambda h(w_{s+1})$ for all $x\in W$ and $\lambda\in \R$ was wrong and the statement follows.
\end{proof}

\begin{rem}
(i): Removing some set of measure zero is necessary in general. We give a concrete example for the fiber body studied in \cref{sec:ternary_quartics}. A proof with a precise statement is contained in \cref{ex:1-dim-vanishing}.

There exists a thin subset $\mathcal{T}\subset\Sigma_{3,4}$ such that for every $f\in\mathcal{T}$ the \gs\ $\gram(f)$ has a face of dimension 1. Pick one such $f$ and a general point $w$ in the normal cone of this face. The face on $\Sigma_\mu K$ in direction $w$ will be an extreme point and not a face of positive dimension.

(ii): The two easiest cases included in \cref{thm:dimension_face_fb} are when the face $S_u^w$ has either dimension 0 or dimension $n-1$ for almost all $u$. In the first case almost all subspaces $A_u$ are zero. Then the same holds for their sums. In the second case almost all are hyperplanes with orthogonal complement spanned by $w$, hence the same holds for their sum as well.
The first case is also included in \cref{prop:fb_extreme_points}.
\end{rem}

We now prepare the proof of \cref{thm:normal_cone_fb}.

\begin{lem}
\label{lem:basis_extension}
Let $0\neq w,w'\in W$ such that $S_u^w\neq S_u^{w'}$ for all $u\in\oo$ where $\oo\subset V$ has non-zero measure. Then we can find $w_1,\dots,w_s\in W$ such that $w,w_1,\dots w_s$ is a basis of $W$ and $S_u^{w,w_1,\dots,w_s}\neq S_u^{w',w_1,\dots,w_s}$ for all $u\in\oo'\subset\oo$ where $\oo'$ has non-zero measure.
\end{lem}
\begin{proof}
If the two faces $S_u^{w,w'}$ and $S_u^{w,-w'}$ do not intersect, the same holds for $S_u^{w,-w'}$ and $S_u^{w',-w'}$: Indeed, assume $x\in S_u^{w,-w'}\cap S_u^{w',-w'}$. Then $\bil{x}{w}$ is the minimal value of $\bil{\cdot}{w}$ on $S_u$ and $\bil{x}{w'}$ is the minimal value of $\bil{\cdot}{w'}$ on $S_u$. Hence, $\bil{x}{w'}$ is also the minimal value of $\bil{\cdot}{w'}$ on $S_u^w$, i.e. $x\in S_u^{w,w'}$. Thus, $x\in S_u^{w,w'}\cap S_u^{w,-w'}=\emptyset$.
Thus we may take $w_1=-w'$ and complete to a basis of $W$.

Assume that $S_u^{w,w'}$ and $S_u^{w,-w'}$ do intersect, then $\bil{\cdot}{w'}$ and $\bil{\cdot}{-w'}$ minimize on a common point of $S_u^w$ and therefore they have constant value on all of $S_u^w$. Either $S_u^{w'}$ then contains all of $S_u^w$ or does not intersect $S_u^w$ at all. In the second case we are done by choosing any basis. In the first case $S_u^{-w'}$ intersects $S_u^{w'}$ in at least $S_u^w$. But then $\bil{\cdot}{w'}$ and $\bil{\cdot}{-w'}$ have the same value on these faces and thus $\bil{\cdot}{w'}$ has constant value on all of $S_u$ which means $S_u$ is contained in the orthogonal complement of $\spn(w')$. Since $\pi$ is surjective, this is the case only for a subset of measure zero which we may remove.

Therefore, we may take $w_1=w'$ and complete to a basis of $W$. For $\oo'$ we remove all $u\in\oo$ from $u$ such that $S_u$ is contained in the orthogonal complement of $\spn(w')$.
\end{proof}

\begin{rem}
We do not know if all faces of the fiber body are exposed faces. This does hold for spectrahedra, hence it might also hold for the fiber body in this case.
\end{rem}

\begin{dfn}
For any closed convex set $C\subset\R^n$ and any face $F\subset C$, the \todfn{normal cone} of $F$ is
\[
\normal_C(F)=\{u\in\R^n\colon \bil{u}{x}\le \bil{u}{y}\,\all x\in F, y\in C\}.
\]
For any point $x\in C$ we define the normal cone $\normal_C(x)$ as the normal cone at $F$ where $F$ is the supporting face of $x$, i.e. the smallest face of $C$ containing $x$.
\end{dfn}

Up to fixing a scalar product and identifying $\R^n$ with its dual space, this is the same as $\{l\in(\R^n)^\du\colon l(x)\le l(y)\,\all x\in F, y\in C\}$.

\begin{thm}
\label{thm:normal_cone_fb}
Let $0\neq w\in W$ and let $F\subset\Sigma_\pi K$ be the face in direction $w$. Then the normal cone of $F$ is given by
\[
\normal_{\Sigma_\pi K}(F)=\overline{\bigcup_\oo\bigcap_{u\in\oo} \normal_{S_u}(S_u^w)}
\]
where the union is over all subsets $\oo\subset V$ such that its complement $\oo^c$ has measure zero.
\end{thm}
\begin{proof}
Let $F_u:=S_u^w$.
For every $\oo$ the intersection of the normal cones is contained in the normal cone at $F$ by \cref{lem:faces_fb}. Hence
\[
\normal_{\Sigma_\pi K}(F)\supset\bigcup_\oo\bigcap_{u\in\oo} \normal_{S_u}(S_u^w).
\]
As the left-hand-side is closed, we may also take the closure of the right-hand-side.

Assume there is $w'$ not in any intersection but in the relative interior of the normal cone at $F$.
Then there exists a set $\oo$ with non-zero measure such that for every $u\in\oo:$ $S_u^{w'}\neq F_u$. By \cref{lem:basis_extension} we can find $w_1,\dots,w_s\in W$ such that $S_u^{w',w_1,\dots,w_s}\neq F_u^{w_1,\dots,w_s}$ are extreme points for all $u\in \oo'\subset \oo$ with $\oo'$ having non-zero measure.
Thus, $(\Sigma_\pi K)^{w,w_1,\dots,w_s}\neq (\Sigma_\pi K)^{w',w_1,\dots,w_s}$ as the section defining extreme points is basically unique by \cref{prop:fb_extreme_points}. But this means $F\neq (\Sigma_\pi K)^{w'}$ contradicting the fact that $w'$ is in the normal cone of $F$.
Therefore,
\[
\interior(\normal_{\Sigma_\pi K}(F))\subset\bigcup_\oo\bigcap_{u\in\oo} \normal_{S_u}(S_u^w)
\]
and by closing both sides the result follows.
\end{proof}

\begin{rem}
\cref{thm:normal_cone_fb} can also be read as follows: Let $N$ be the normal cone at $F$ and $N_u$ the normal cone at $S_u^w$. If $w\notin N_u$ for all $u\in\oo$ where $\oo$ is a set with non-zero measure, then $w\notin N$.
Conversely, if $w\in N_u$ for almost all $u$, then $w\in N$.

Note that $w$ being contained in $N_u$ for all $u$ in a set with non-zero measure does not allow any conclusion.
\end{rem}

\begin{rem}
Informally speaking, \cref{thm:dimension_face_fb} says that faces of the fiber body want to be as large as possible. Whereas \cref{thm:normal_cone_fb} means that normal cones want to be as small as possible and make the boundary look as smooth as possible.
\end{rem}

\begin{prop}
\label{prop:continuous_points_nc}
Let $S:=S_u, u\in V$ be a spectrahedron with algebraic boundary defined by a polynomial $F$, and let $r\in\N$. Assume
\begin{enumerate}
\item there are only finitely many rank $r$ points in $S$,
\item there are no points of rank $<r$,
\item the intersection of $\aff(S)\otimes_\R \C$ with the set of rank $r$ matrices is transversal in all points.
\end{enumerate}
Then the rank $r$ points depend continuously on $u$ in a neighbourhood of $u$.

If moreover, at any of the psd rank $r$ points the lowest degree term of the Taylor expansion of $F$ is irreducible, and for all $u'$ in a neighbourhood of $u$ this term has the same degree and the number of rank $r$ psd points in constant, then the normal cones at the rank $r$ points also depend continuously on $u$.
\end{prop}
\begin{proof}
The rank $r$ points on $S$ are given as the real psd matrices of the intersection $(\aff(S)\otimes \C)\cap \{\mathrm{rk}\le r\}$. Since the intersection is transversal the number of intersection points is constant in a neighbourhood of $u$. Hence, they depend continuously on $u$ by the implicit function theorem.

The algebraic boundary of the normal cone at any rank $r$ point is given as the dual variety to the lowest degree part of the Taylor expansion of $F$ at the point. This lowest degree is constant by assumption and the coefficients depend continuously on $u$. Therefore, the same holds for the boundary of the normal cone.
\end{proof}

\section{Faces of Gram spectrahedra}
\label{sec:faces_of_gram_spectrahedra}

In \cref{sec:binary_sextics} and \cref{sec:ternary_quartics} we study the fiber body of two families of \gsa. In this section we provide the necessary background on \gsa\ and prove a formula to calculate the dimensions of normal cones (\cref{thm:normal_cones}).

Let $n, d\in\N$. We write $\Rx_d=\R[x_1,\dots,x_n]_d$ for the vector space of homogeneous degree $d$ polynomials in the variables $x_1,\dots,x_n$. Consider the cone of all sums of squares $\Sigma_{n,2d}\subset\Rx_{2d}=\R[x_1,\dots,x_n]_{2d}$, that is the set of all forms in $\Rx_{2d}$ such that there exist $f_1,\dots,f_r\in\Rx_d$ with $f=\sum_{i=1}^r f_i^2$.
Let $N:=\dim \Rx_d$ and consider the Gram map $\mu\colon \Sym^N\to \Rx_{2d},\, G\mapsto XGX^T$ where $X$ is the vector containing the monomial basis of $\Rx_d$ in some fixed order.
As this is a linear map we can use the fiber body construction from \cref{sec:fb_spec}.

In \cref{sec:binary_sextics} and \cref{sec:ternary_quartics} we use our findings from the last section to study the boundary of the fiber body for some families of Gram spectrahedra. We therefore recall what is known about the boundary of such \gsa. All results cited are taken from \cite{Scheiderer18} and \cite{Vill21}.

Instead of symmetric matrices we are going to work with symmetric tensors which is the same up to fixing a basis. We denote by $\sy{\Rx_d}$ the second symmetric power of $\Rx_d$. Let $\theta\in\sy{\Rx_d}$ with $\theta=\sum_{i=1}^r a_i p_i\otimes p_i,\, a_i\in\R$. The tensor $\theta$ is called positive semidefinite ($\theta\succeq 0$) if for every $\lambda\in(\Rx_d)^\du$, $\sum_{i=1}^r a_i\lambda(p_i)^2\ge 0$. We denote by $\im(\theta)$ the image of $\theta$ which is defined as the image of the linear map $\Rx_d^\du\to \Rx_d,\, \lambda\mapsto \sum_{i=1}^r a_i \lambda(p_i) p_i$.

Let $f\in\Sigma_{n,2d}$. We write $\gram(f)=\{\theta\in\sy^+\Rx_d\colon \mu(\theta)=f\}$ for the Gram spectrahedron of $f$ where $\mu\colon \sy\Rx_d\to \Rx_{2d}$ is the Gram map given by $p\otimes p\mapsto p^2$. We prefer to use this coordinate-free setup but will rather freely change our point of view when convenient.
We use the same notation as in \cref{sec:fb_spec} and denote by $W$ the kernel of $\mu$ and by $V$ its orthogonal complement with respect to the apolarity pairing:
For $i=1,\dots,n$ define the differential operator $\partial_i\vcentcolon = \frac{\partial}{\partial x_i}$ and $\partial\vcentcolon = (\partial_1,\dots,\partial_n)$, $\partial^\alpha=\partial^{\alpha_1}\cdot\dots\cdot\partial^{\alpha_n}$ for $\alpha\in\Z^n_+$. For $f=\sum_\alpha c_\alpha x^\alpha\in \Rx_d$ define $f(\partial)\vcentcolon = \sum_\alpha c_\alpha \partial^\alpha$. For every $m\ge 0$ we then have the following bilinear form on $\Rx_d$:
\[
\bil{f}{g}\vcentcolon =\frac{1}{d!}f(\partial)(g)=\frac{1}{d!}g(\partial)(f),\ \all f,g\in \Rx_d.
\]
This extends to a scalar product on $\sy\Rx_d$: Indeed, let $p\otimes p,q\otimes q\in\sy{\Rx_d}$, then define $\bil{p\otimes p}{q\otimes q}:= \bil{p}{q}^2$ on rank 1 tensors and extend bilinearly. One easily checks that this again defines a scalar product.

We will mostly consider $\gram(f)$ as a subset of $W$. Note that by \cref{lem:correspondence_functionals} it does not matter if we work inside $W$ or one of its translates.

Let $U\subset\Rx_d$ be a subspace. We denote by $U^2$ the subspace of $\Rx_{2d}$ spanned by all products $pq$ with $p,q\in U$, and by $\Sigma U^2$ the cone of sums of squares of $U$, i.e. the set of all $f\in\Rx_{2d}$ such that there exist $f_1,\dots,f_r\in U$ with $f=\sum_{i=1}^r f_i^2$. As the cone $\Sigma_{n,2d}$, the cone $\Sigma U^2$ is a full-dimensional, closed, convex, pointed cone inside the space $U^2$. In particular, $\dim U^2=\dim\Sigma U^2$. 
For a form $f\in\Sigma_{n,2d}$ lying in the relative interior of $\Sigma U^2$ is equivalent to having an sos representation $f=\sum_{i=1}^r f_i^2$ with $f_1,\dots,f_r$ a basis of $U$, which again is equivalent to having a Gram tensor with image $U$.

\begin{rem}
If we work with matrices instead of tensors we use the standard trace bilinear form on the set of symmetric matrices. Note that the trace scalar product and the apolarity pairing do not translate perfectly. The first one gives the standard scalar product on $\R^{\binom{N}{2}}$ whereas in the second one monomials are weighted by a multifactorial and by the degree. 

Especially in \cref{sec:ternary_quartics} we will stick to the tensor point of view as characterizations of certain faces are more natural.
\end{rem}

\begin{cor}
Let $U\subset\Rx_d$ be a subspace and let $u=v+w\in V\oplus W$ be any tensor with image $U$. Then for every $f\in\Sigma U^2$ the face of $\gram(f)$ in direction $u$ is given by $\{G\in\gram(f)\colon \im(G)\subset\im (u)\}\neq\emptyset$.

Especially, for every $f\in\interior\Sigma U^2$ the dimension of the face of $\gram(f)$ in direction $u$ is $\dim\sy{U}-\dim U^2$.
\end{cor}
\begin{proof}
This is \cref{prop:dimension_face_spec} in the case of \gsa. A proof can also be found in \cite[section 2]{Scheiderer18}
\end{proof}

\begin{rem}
Note that whenever $U$ is such that $U^2=\Rx_{2d}$ there exists a \sa\ set with non-empty interior, namely $\interior\Sigma U^2$, such that for every $f\in \interior\Sigma U^2$ the face on $\gram(f)$ in direction $u$ has the same dimension $\dim\sy{U}-\dim U^2=\binom{\dim U+1}{2}-\binom{n+2d-1}{2d}$.
\end{rem}

To understand the boundary of $\Sigma_\mu K$ we do not only want to consider the dimensions of the faces but also the normal cones at each point. Since the fiber bodies might not even be \sa\ we especially cannot stratify their boundary by the rank as in the case of spectrahedra. As a substitute we will use the dimension of the normal cones. This is especially useful in the cases of Gram spectrahedra we are interested in as the rank of points on the boundary determines also the dimension of the normal cone.

For normal cones of Gram spectrahedra we have the following theorems.

\begin{thm}[{\cite[Theorem 2.9.]{st2015}}]
\label{thm:st_normal_cones}
Let $f\in \interior\Sigma_{n,2d}$ and $\theta\in\gram(f)$. Write $U\vcentcolon = \im(\theta)$ and $r\vcentcolon = \dim U$. Then
\[
\dim\normal_{\gram(f)}(\theta)=\dim(\sy{\Rx_{d}})-\dim(\ker\mu\cap \Sym(U\otimes \Rx_d))
\]
where $\Sym$ is the symmetrization map $\Rx_d\otimes \Rx_d\to \sy \Rx_d$.
\end{thm}

\begin{thm}
\label{thm:normal_cones}
Let $f\in \interior\Sigma_{n,2d}$ and $\theta\in\gram(f)$. Write $U\vcentcolon = \im(\theta)$ and $r\vcentcolon = \dim U$. Then
\[
\dim\normal_C(\theta)=\dim(\sy \Rx_d)-\crk(\dphi(U))+\binom{r}{2}.
\]
where $\dphi(U)\colon \Rx_d^{r}\to \Rx_{2d},\, (p_1,\dots,p_r)\mapsto 2\sum_{i=1}^r p_iq_i$ with $q_1,\dots,q_r$ a basis of $U$.
\end{thm}
\begin{proof}
We write $V\vcentcolon \Rx_{2d}$. Using \cref{thm:st_normal_cones} we need to show
\[
\dim(\ker\mu\,\cap\, \Sym(U\otimes V))=\crk(\dphi(U))-\binom{r}{2}.
\]
Since $UV$ is the image of $\dphi(U)$, we have 
\[
\dim UV=r\cdot\dim V-\crk(\dphi(U)). 
\]
The dimension of $\Sym(U\otimes V)$ is given by $\binom{\dim V+1}{2}-\binom{\dim V-r+1}{2}$, since $\Sym(U\otimes V)\cong \sy{V}/\sy{(U^\perp)}$.
The map $\mu\vert_{\Sym(U\otimes V)}\colon \Sym(U\otimes V)\to UV$ is surjective, hence we calculate
\begin{align*}
\dim(\ker\mu\,\cap\,& \Sym(U\otimes V))=\\
&=\dim \Sym(U\otimes V)-\dim UV\\
&=\binom{\dim V+1}{2}-\binom{\dim V-r+1}{2}-r\dim V+\crk(\dphi(U))\\
&=\crk(\dphi(U))-\binom{r}{2}.
\end{align*}
\end{proof}

This gives the dimension of all normal cones we need in the next sections.

\begin{rem}
The dimension of the normal cone only depends on $\im\theta$ and not on $\theta$ itself.
\end{rem}

\section{Binary sextics}
\label{sec:binary_sextics}

We consider the case $n=2,\, d=3$ of binary sextics. The \gs\ of a general psd binary sextic is a 3-dimensional convex set. This is the first case where the \gs\ is not a polytope. Indeed, if $d=1$, i.e. we consider binary quadratic forms, the \gs\ is a single point. If $d=2$ the \gs\ is a line segment, in particular it is a polytope. Since the fiber body has the same dimension, it is then also a line segment itself.
The situation become much more interesting in the next case $d=3$. We show that the boundary of the fiber body consists only of extreme points and exactly one vertex, i.e. an extreme point with a full-dimensional normal cone. This is very unexpected as on the one hand almost all \gsa\ of psd binary sextics have exactly four vertices and on the other hand by \cref{thm:normal_cone_fb} we expect the boundary to be as smooth as possible and every point having a 1-dimensional normal cone.

The general situation regarding \gsa\ in this case is the following which can be found in \cite{Scheiderer18}. The dimension of the normal cones follows from \cref{thm:normal_cones}.

\begin{thm}
Let $f\in\interior\Sigma_{2,6}$. The Gram spectrahedron $\gram(f)\subset \sy^+{\Rx_3}$ has dimension 3. If $f$ is chosen with distinct zeros, there are exactly 4 rank 2 Gram tensors on the boundary of $\gram(f)$, each with a 3-dimensional normal cone. The rest of the boundary consists of rank 3 Gram tensors which are extreme points and have 1-dimensional normal cones.

For special choices of $f$ the number of rank 2 Gram tensors can also be lower than 4.
\end{thm}

In this section we show the following.

\begin{thm}
\label{thm:binary_sextics_fb}
The fiber body $\Sigma_\mu K$ has dimension 3 and every point on its boundary is an extreme point. There is exactly one point $\Theta$ that has a 3-dimensional normal cone, every other point has a 1-dimensional normal cone. This point $\Theta$ corresponds to the (essentially unique) section $\gamma$ with $\gamma(f)=\theta$ where $\theta$ is the rank 2 point corresponding to the factorization of $f$ where the zeros of one factor all have positive imaginary part (and the zeros of the other all have negative imaginary part).
Moreover, the polynomial inequalities defining the normal cone are known, especially the 3-dimensional normal cone is \sa.
\end{thm}
\begin{proof}
From \cref{prop:fb_extreme_points} we immediately see that every point on the boundary is an extreme point since this is true for the fibers. The rest will follow from \cref{thm:binary_sextics_3_dim_nc} and \cref{prop:binary_sextics_1_dim_nc}.
\end{proof}

\begin{rem}
The fact that there is a 3-dimensional normal cone means that for any direction/linear functional in the normal cone, the optimal solution when minimizing this linear functional over \textit{any} \gs\ of a psd binary sextic with distinct zeros, will be a rank 2 Gram tensor.
\end{rem}

\begin{rem}
In this case all faces of the fiber body are exposed since every exposed face is an extreme point.
\end{rem}

For the rest of this section the goal is to show the statement about the normal cones.
As a basis for $\Rx_3$ we choose $x^3,x^2y,xy^2,y^3$ in this order. The matrices
\[
R_1=\left(\begin{array}{rrrr}
0 & 0 & 1 & 0 \\
0 & \minus 2 & 0 & 0 \\
1 & 0 & 0 & 0 \\
0 & 0 & 0 & 0
\end{array}\right),
R_2=\left(\begin{array}{rrrr}
0 & 0 & 0 & 1 \\
0 & 0 & \minus 1 & 0 \\
0 & \minus 1 & 0 & 0 \\
1 & 0 & 0 & 0
\end{array}\right),
R_3=\left(\begin{array}{rrrr}
0 & 0 & 0 & 0 \\
0 & 0 & 0 & 1 \\
0 & 0 & \minus 2 & 0 \\
0 & 1 & 0 & 0
\end{array}\right)
\]
then form an orthogonal basis of $W=\ker\mu$ with respect to the trace scalar product. We identify $W$ and $V=W^\du$ via this scalar product.

We denote by $K'$ the set of all positive definite binary sextics with distinct zeros.
Let $f\in K'$. Any factorization $f=g\ol g$ gives rise to a Gram tensor of $f$ corresponding to the sos representation $f=\left(\frac{1}{2}(g+\ol g)\right)^2+\left(\frac{1}{2i}(g-\ol g)\right)^2$. Let $\theta_f$ be the Gram tensor corresponding to the factorization where all zeros of $g$ have positive imaginary part, and let $\theta_2,\theta_3,\theta_4$ be the other 3 Gram tensors of rank 2.

\begin{lem}
\label{lem:binary_sextics_ex_non_intersection}
Consider the two psd binary sextics with zeros
\begin{align*}
&\{1+6i, 1-6i, 2+5i, 2-5i, 3+4i, 3-4i\},\\
&\{-4+2i, -4-2i, -2+2i, -2-2i, -1+2i, -1-2i\}.
\end{align*}
The six normal cones at the points $\theta_2,\theta_3,\theta_4$ do not intersect.
\end{lem}
\begin{proof}
For both binary forms we calculate the three rank 2 Gram matrices $\theta_2,\theta_3,\theta_4$ together with the three degree 2 parts of the form $\theta_i+\sum_{i=1}^3 \lambda_i R_i\in\R[\lambda_1,\lambda_2,\lambda_3]$ (this is the Taylor expansion of the algebraic boundary of the \gs\ around $\theta_i$). The normal cone is then given by inverting the symmetric matrix obtained from these quadratic forms. We then check with \texttt{Mathematica} that there is no non-zero point contained in two different normal cones. Note that it is not necessary to determine which side of the double-sided cones defined by the quadratic forms is in fact the normal cone. The statement that two different cones do not intersect even holds for the double-sided cones. For completeness sake we include the symmetric matrices defining the normal cones. For the first sextic they are 
\[
\left(\begin{array}{rrr}
-5 & 11 & 101 \\
11 & -149 & 277 \\
101 & 277 & -4017
\end{array}\right),
\left(\begin{array}{rrr}
-3 & 1 & 115 \\
1 & -107 & 135 \\
115 & 135 & -4567
\end{array}\right),
\left(\begin{array}{rrr}
-7 & 19 & 55 \\
19 & -163 & 545 \\
55 & 545 & -4683
\end{array}\right)
\]
for the second one we get
\[
\left(\begin{array}{rrr}
-1 & -1 & 3 \\
-1 & -2 & 7 \\
3 & 7 & 20
\end{array}\right),
\left(\begin{array}{rrr}
-1 & -2 & 0 \\
-2 & -10 & -25 \\
0 & -25 & -100
\end{array}\right),
\left(\begin{array}{rrr}
-1 & -4 & -12 \\
-4 & -14 & -47 \\
-12 & -47 & -220
\end{array}\right).
\]
\end{proof}

\begin{thm}
\label{thm:binary_sextics_3_dim_nc}
The normal cone of $\Theta$ is given as
\[
S:=\{(\lambda_1,\lambda_2,\lambda_3)\colon\lambda_2^2\le 4\lambda_1\lambda_3,\, 0\le \lambda_1,\lambda_3\}
\]
in the basis $R_1,R_2,R_3$.
\end{thm}

\begin{lem}
\label{lem:binary_sextics_rk_1_completion}
Let $0\neq w\in W\setminus S$. Then there exist $q\in\Rx_3$ and $v\in V$ such that $q\otimes q=v+w$.
\end{lem}
\begin{proof}
Let $w=\sum_{i=1}^3 \frac{\lambda_i}{\norm{R_i}}R_i$ and let $q=a_1x^3+a_2x^2y+a_3xy^2+a_4y^3$. Then via direct calculation we see that
$q\otimes q= v + w$ ($v\in V,\, w\in W$) ($v$ is the unique point in $V$ that maps to $q^2$ since $\mu(q\otimes q)=q^2$ and $\mu(w)=0$) if
\[
\lambda_1=2(-a_2^2+a_1a_3),\quad \lambda_2=2(-a_2a_3+a_1a_4),\quad \lambda_3=2(a_3^2+a_2a_4).
\]
Hence, given $w$, or equivalently $\lambda_1,\lambda_2,\lambda_3\in\R$, we want to find $a_1,\dots,a_4\in\R$ as above. We use \texttt{Mathematica} to solve this over the reals and get the complement of $S$.
\end{proof}

\begin{rem}
If one wants to avoid the use of \texttt{Mathematica}, this can also be done by hand as one can explicitly write down solutions. In these the only term that needs closer inspection is
\[
\sqrt{-4 \lambda_1^2 \lambda_3 + a_2^2(\lambda_2^2 - 4 \lambda_1 \lambda_3)}
\]
which is easily done. However, depending on whether $\lambda_i=0$ for some $i\in\{1,2,3\}$ there are several case distinctions.
\end{rem}

\begin{prop}
\label{prop:binary_sextics_1_dim_nc}
Let $0\neq w\in W\setminus S$. The face of $\Sigma_\mu K$ in direction $w$ has a 1-dimensional normal cone.
\end{prop}
\begin{proof}
By \cref{lem:binary_sextics_rk_1_completion} we can find $q\in\Rx_3,\, v\in V$ such that $q\otimes q=v+w$. Let $U=\spn(q)^\perp$. Then for every $f\in\interior\Sigma U^2$ the face of $\gram(f)$ in direction $w$ is a Gram tensor with image $U$, especially the normal cone has dimension 1.
For a general choice of $f\in\Sigma U^2$ the normal cones of the rank 2 points depend continuously on $f$ by \cref{prop:continuous_points_nc} as $f$ has only simple zeros and hence the face in direction $w$ has a 1-dimensional normal cone in a neighbourhood of $f$. Hence, the normal cone of $\Sigma_\mu K$ in direction $w$ has dimension 1 by \cref{thm:normal_cone_fb}.
\end{proof}

\begin{proof}[Proof of \cref{thm:binary_sextics_3_dim_nc}]
Let $w\in S$. Then $w$ does not span the normal cone of a rank 3 point on any \gs\ since we cannot complete to a psd rank 1 tensor by the proof of \cref{lem:binary_sextics_rk_1_completion}. Therefore, on all of them the face is a rank 2 extreme point. 
By \cref{lem:binary_sextics_ex_non_intersection} there exists $f\in K'$ such that the face in direction $w$ is $\theta_f$.
Since the set $S$ is convex, especially connected, and the same holds for $K'$ where normal cones at rank 2 points depend continuously on $f$, it follows that $S$ is contained in the normal cone at $\theta_f$ for every $f$. (Note that the normal cone at $\theta_f$ is usually larger.)
\end{proof}

\begin{rem}
On the left in \cref{fig:binary_sextics} we see the 3-dimensional fiber body in the case of binary sextics. This picture was generated using \texttt{Julia} \cite{beks2017} as follows: We first sample 50 random psd binary sextics and 10000 random points in $W$. Solving the SDP on every \gs\ and every direction, we find 10000 points on the boundary in all 50 cases. For any direction, the face on the fiber body in this direction is then taken as the average of the 50 points. The extreme point at the bottom left of the picture is the single point on the boundary with a 3-dimensional normal cone.

On the right in \cref{fig:binary_sextics} we see the \gs\ of the form $x^6+y^6$. According to experiments in \texttt{Julia} the probability is rather high that the \gs\ of a random psd binary sextic has the other three rank 2 points located similarly to this \gs, i.e. the normal cones intersecting. This can also still be seen on the boundary of the fiber body as these three directions look more 'pointed' even though they have 1-dimensional normal cones. 
\end{rem}

\begin{rem}
According to experiments in \texttt{Mathematica}, the 3-dimensional normal cone is in fact a limit and not the intersection of finitely many normal cones. It seems that $w=R_3$ is contained in the interior of the normal cone for every $f$ but is contained in the boundary of the normal cone on the fiber body. This makes it even more surprising that the normal cone at $\Theta$ has such an easy \sa\ description.

We do not know whether the fiber body itself is a \sa\ set. 
\end{rem}

\begin{figure}
    \begin{center}
      \begin{minipage}{.48\textwidth}
      \includegraphics[scale=0.35]{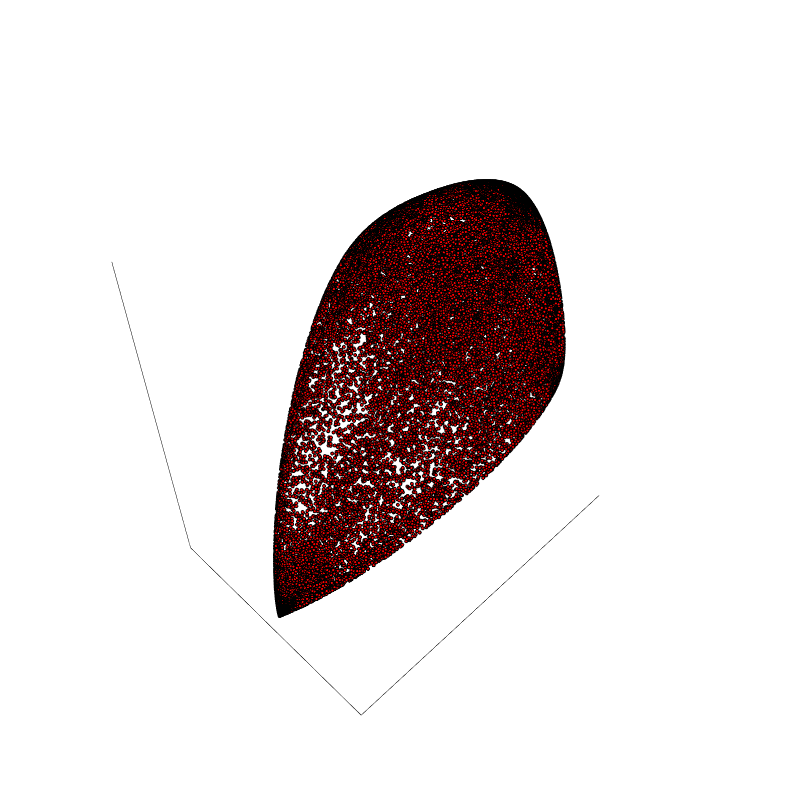}
      \end{minipage}
      \begin{minipage}{.48\textwidth}
      \includegraphics[scale=0.35]{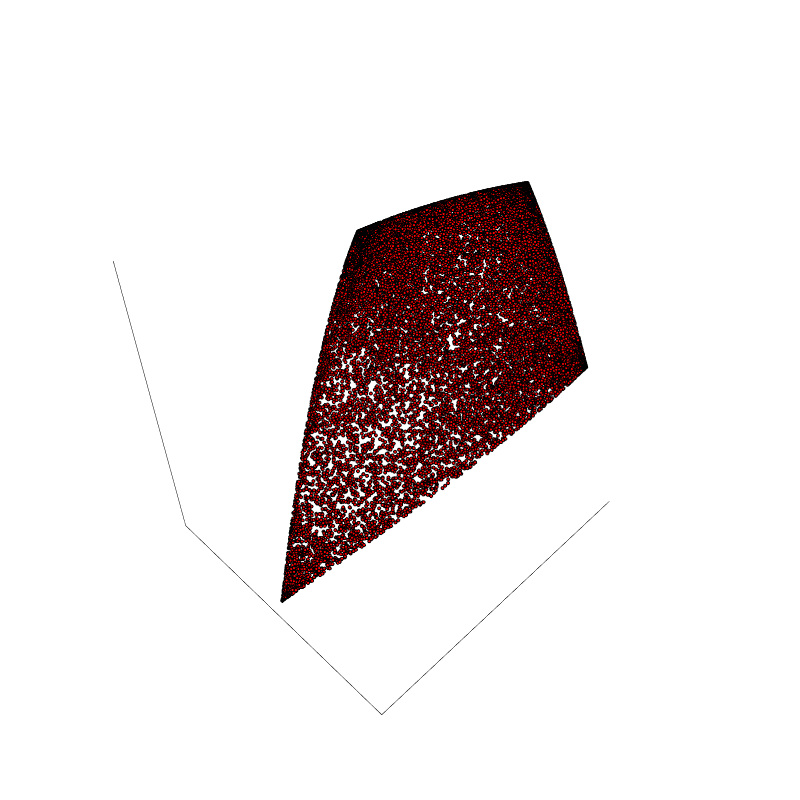}
      \end{minipage}
      \caption{Left: the fiber body in the case of binary sextics. Right: the \gs\ of the form $x^6+y^6$.}
      \label{fig:binary_sextics}
    \end{center}
  \end{figure}

\section{Ternary quartics}
\label{sec:ternary_quartics}

In this section we study the fiber body in the case of ternary quartics, i.e. $n=3,\, d=2$. \gsa\ in this case generally are of dimension 6 and have positive dimensional faces as well as a larger number of possible dimensions of normal cones. Moreover, it is the first case where the Pataki range has length 3, i.e. extreme points on the \gs\ of a general psd ternary quartic may have rank 3,4 or 5.
We show that the fiber body has a rather different facial structure than a general \gs. On the one hand, the dimensions of faces grow, as the 2-dimensional family of 2-dimensional faces is replaced by a 2-dimensional family of 3-dimensional faces. On the other hand, every face of the fiber body has a 1-dimensional normal cone whereas \gsa\ have normal cones of dimensions 1,3 and 6 (\cref{thm:ternary_quartics_fb}).

For any $f\in\Rx_{2d}$ and $r,s\in\N_0$ we denote by $S_f(r,s)$ the semi-algebraic set of all Gram tensors of $f$ of rank $r$ that lie in the relative interior of a face of dimension $s$.

\begin{thm}[\cite{Vill21}, \cite{prss2004}]
Let $f\in\Sigma_{3,4}$ generic. The Gram spectrahedron $\gram(f)\subset\sy^+{\Rx_2}$ has dimension 6. A general point on the boundary is a rank 5 extreme point. Every rank 5 extreme point has a 1-dimensional normal cone.
The semi-algebraic set $S_f(5,2)$ has dimension 4 and every point in it also has a 1-dimensional normal cone.
The set $S_f(4,0)$ has dimension 3 and every point has a 3-dimensional normal cone. Moreover, there are exactly 8 rank 3 Gram matrices each with a 6-dimensional normal cone.

There are no Gram tensors of rank lower than 3 and there are no other faces than the ones mentioned above (for generic $f$).
\end{thm}

\begin{prop}[{\cite[proof of Prop 5.2.]{Vill21}}]
\label{prop:ternary_quartics_rk_5_faces}
Let $f\in\Sigma_{3,4}$ smooth and $\theta\in\gram(f)$ be a rank 5 tensor with $U:=\im(\theta)$. Then
\begin{enumerate}
\item $\theta$ is an extreme point if and only if $\dim U^2=15$, i.e. $U^2=\Rx_4$, equivalently $U^\perp$ is spanned by a quadratic form of rank 3, and
\item $\theta$ is contained in the relative interior of a 2-dimensional face if and only if $\dim U^2=13$ if and only if $U^\perp$ is spanned by a quadratic form of rank 2.
\end{enumerate}
\end{prop}

\begin{rem}
If $U^\perp$ is spanned by a quadratic form of rank 1, then elements in $U$ have a common real zero and hence the same holds for $f$ which is impossible for smooth $f$.
\end{rem}

In the case of ternary quartics we order the monomial basis of $\Rx_2$ as $(x^2,y^2,z^2,xy,xz,yz)$. With this choice, a basis of $W$ is given by
\[
R_1=\left(\begin{array}{rrrrrr}
0 & 1 & 0 & 0 & 0 & 0 \\
1 & 0 & 0 & 0 & 0 & 0 \\
0 & 0 & 0 & 0 & 0 & 0 \\
0 & 0 & 0 & \minus 2 & 0 & 0 \\
0 & 0 & 0 & 0 & 0 & 0 \\
0 & 0 & 0 & 0 & 0 & 0
\end{array}\right),
R_2=\left(\begin{array}{rrrrrr}
0 & 0 & 1 & 0 & 0 & 0 \\
0 & 0 & 0 & 0 & 0 & 0 \\
1 & 0 & 0 & 0 & 0 & 0 \\
0 & 0 & 0 & 0 & 0 & 0 \\
0 & 0 & 0 & 0 & \minus 2 & 0 \\
0 & 0 & 0 & 0 & 0 & 0
\end{array}\right),
R_3=\left(\begin{array}{rrrrrr}
0 & 0 & 0 & 0 & 0 & 0 \\
0 & 0 & 1 & 0 & 0 & 0 \\
0 & 1 & 0 & 0 & 0 & 0 \\
0 & 0 & 0 & 0 & 0 & 0 \\
0 & 0 & 0 & 0 & 0 & 0 \\
0 & 0 & 0 & 0 & 0 & \minus 2
\end{array}\right)
\]
\[
R_4=\left(\begin{array}{rrrrrr}
0 & 0 & 0 & 0 & 0 & 1 \\
0 & 0 & 0 & 0 & 0 & 0 \\
0 & 0 & 0 & 0 & 0 & 0 \\
0 & 0 & 0 & 0 & \minus 1 & 0 \\
0 & 0 & 0 & \minus 1 & 0 & 0 \\
1 & 0 & 0 & 0 & 0 & 0
\end{array}\right),
R_5=\left(\begin{array}{rrrrrr}
0 & 0 & 0 & 0 & 0 & 0 \\
0 & 0 & 0 & 0 & 1 & 0 \\
0 & 0 & 0 & 0 & 0 & 0 \\
0 & 0 & 0 & 0 & 0 & \minus 1 \\
0 & 1 & 0 & 0 & 0 & 0 \\
0 & 0 & 0 & \minus 1 & 0 & 0
\end{array}\right),
R_6=\left(\begin{array}{rrrrrr}
0 & 0 & 0 & 0 & 0 & 0 \\
0 & 0 & 0 & 0 & 0 & 0 \\
0 & 0 & 0 & 1 & 0 & 0 \\
0 & 0 & 1 & 0 & 0 & 0 \\
0 & 0 & 0 & 0 & 0 &  \minus 1 \\
0 & 0 & 0 & 0 & \minus 1 & 0
\end{array}\right).
\]
Every element of $V$ therefore has the form
\[
A(c_{400},\dots,c_{004})=\left(\begin{array}{rrrrrr}
c_{400} & \frac{1}{3} c_{220} & \frac{1}{3} c_{202} & \frac{1}{2} c_{310} & \frac{1}{2} c_{301} & \frac{1}{4} c_{211} \\[5pt]
\frac{1}{3} c_{220} & c_{040} & \frac{1}{3} c_{022} & \frac{1}{2} c_{130} & \frac{1}{4} c_{121} & \frac{1}{2} c_{031} \\[5pt]
\frac{1}{3} c_{202} & \frac{1}{3} c_{022} & c_{004} & \frac{1}{4} c_{112} & \frac{1}{2} c_{103} & \frac{1}{2} c_{013} \\[5pt]
\frac{1}{2} c_{310} & \frac{1}{2} c_{130} & \frac{1}{4} c_{112} & \frac{1}{3} c_{220} & \frac{1}{4} c_{211} & \frac{1}{4} c_{121} \\[5pt]
\frac{1}{2} c_{301} & \frac{1}{4} c_{121} & \frac{1}{2} c_{103} & \frac{1}{4} c_{211} & \frac{1}{3} c_{202} & \frac{1}{4} c_{112} \\[5pt]
\frac{1}{4} c_{211} & \frac{1}{2} c_{031} & \frac{1}{2} c_{013} & \frac{1}{4} c_{121} & \frac{1}{4} c_{112} & \frac{1}{3} c_{022}
\end{array}\right)
\]
for some choice of $c_{400},\dots,c_{004}\in\R$. Under the Gram map $\mu$ such a matrix maps to the polynomial $f=\sum_{\alpha\in\Z_+^3} c_\alpha \ul x^\alpha\in\Rx$.

In tensor language an orthogonal basis is given by
\begin{align*}
R_1=x^2\otimes y^2-xy\otimes xy,\quad R_2=x^2\otimes z^2-xz\otimes xz,\quad R_3=y^2\otimes z^2-yz\otimes yz\\
R_4=x^2\otimes yz-xy\otimes xz,\quad R_5=y^2\otimes xz-xy\otimes yz,\quad R_6=z^2\otimes xy-xz\otimes yz.
\end{align*}

From now on we will only work with the tensors and denote by $\bil{\cdot}{\cdot}$ the scalar product on $\sy\Rx_2$ introduced in \cref{sec:faces_of_gram_spectrahedra}.

\begin{rem}
For $i=1,2,3$, we have $\bil{R_i}{R_i}=\frac{3}{4}$ and for $i=4,5,6$ we have $\bil{R_i}{R_i}=1$.
\end{rem}

We show the following theorem concerning the boundary structure of $\Sigma_\mu K$.

\begin{thm}
\label{thm:ternary_quartics_fb}
The boundary of $\Sigma_\mu K$ contains a 2-dimensional family of 3-dimensional faces. Every other point on the boundary is an extreme point.
Every face has a normal cone of dimension 1.
\end{thm}

\begin{rem}
As in the case of binary sextics we do not know if the fiber body is \sa.
\end{rem}

First, we show \cref{thm:ternary_quartics_3-dim-face-and-nc} which is the first part of \cref{thm:ternary_quartics_fb} concerning the existence of 3-dimensional faces.

\begin{lem}
\label{lem:ternary_quartics_direction_R1}
Let $w= R_1$. There is a full-dimensional \sa\ set of forms $f\in\Sigma_{3,4}$ such that the face of $\gram(f)$ in direction $w$ has dimension 2.
\end{lem}
\begin{proof}
The face in direction $w$ is the same as the face in direction $\lambda w$ for any $\lambda\in\R_{>0}$.
Let $0\neq q\in\R[x,y]_2$ be a quadratic form. Then
\[
q\otimes q=\theta+\frac{\bil{R_1}{q\otimes q}}{\bil{R_1}{R_1}} R_1=\theta-3\disc(q)R_1
\]
where $\theta\in V$ is the unique Gram tensor with $\mu(\theta)=q^2$.

Let $q\in\R[x,y]_2$ with $\disc(q)<0$ and let $U_q:=\spn(q)^\perp\subset \Rx_2$. As $\rk(q)=2$ we know $\dim (U^2)=13$ by \cref{prop:ternary_quartics_rk_5_faces}. For every $f\in \interior \Sigma U^2$ the face of $\gram(f)$ in direction $\theta-3\disc(q) R_1$ has dimension 2, which is also the face in direction $R_1$. Consider the \sa\ set
\[
X=\bigcup \interior \Sigma U_q^2
\]
where the union is taken over all $q\in\P\R[x,y]_2$ with $\disc(q)<0$. We claim that this set has dimension 15. Assume $f\in\interior \Sigma (U_q)^2\cap\interior\Sigma (U_p)^2$. Then every Gram tensor on the face of $\gram(f)$ in direction $R_1$ has image $U_p$ and $U_q$, hence they are the same and $p=q$ (in $\P\R[x,y]_2$).
As there is a 2-dimensional set of such $q$ and $\dim\Sigma U_q^2=13$, we see that $\dim X=15$. 
\end{proof}

\begin{rem}
The proof also shows that the same holds in direction $-R_1$ by choosing $q\in\P\R[x,y]_2$ with positive discriminant. 

Moreover, note that $\disc(q)=0$ means $q$ is a rank 1 quadratic form and has the form $q=l^2$ for some linear form $l$. This means that every $f\in\Sigma U_q^2$ has a real zero corresponding to $l$ and $\dim (U_q)^2=12$. In this case the tensor $q\otimes q$ is contained in $V$.
\end{rem}

\begin{figure}
    \begin{center}
      \begin{minipage}{.48\textwidth}
      \includegraphics[scale=0.35]{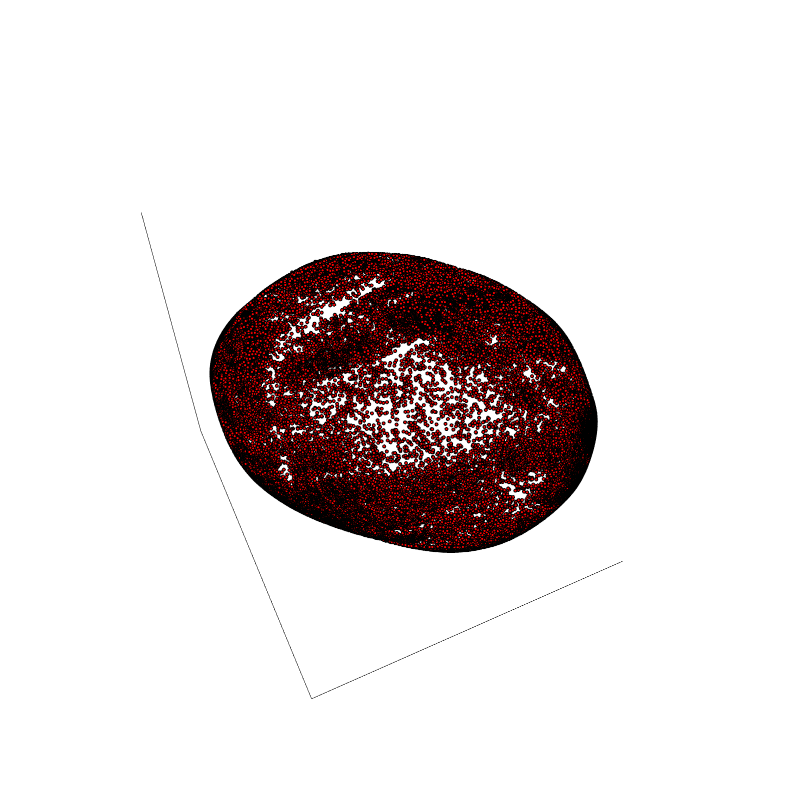}
      \end{minipage}
      \begin{minipage}{.48\textwidth}
      \includegraphics[scale=0.35]{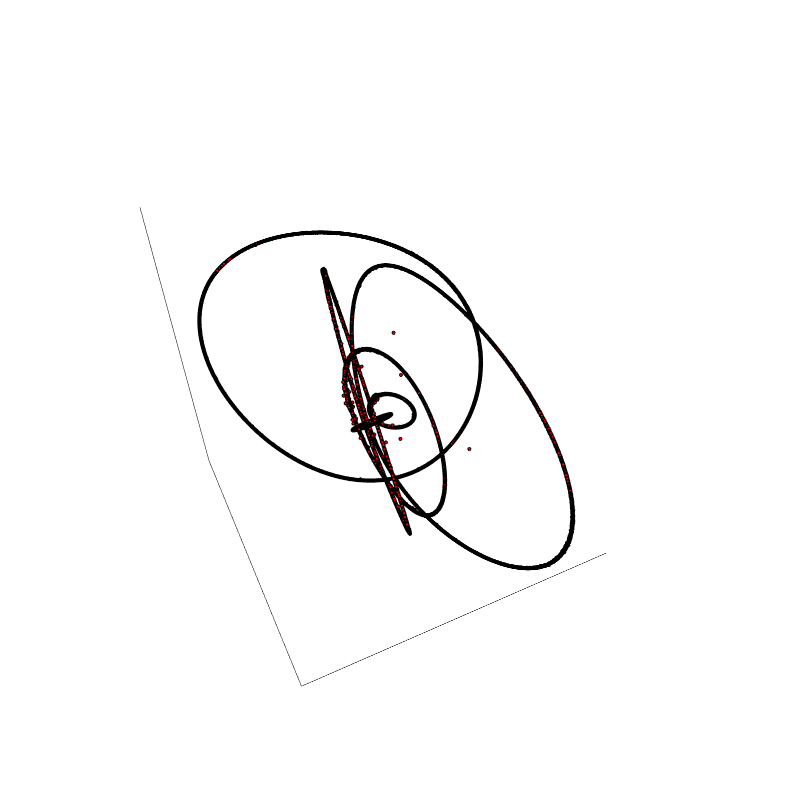}
      \end{minipage}
      \caption{Left: the 3-dimensional face of the fiber body in direction $R_1$. Right: six boundaries of the 2-dimensional face of a \gs\ in direction $R_1$.}
      \label{fig:ternary_quartic}
    \end{center}
  \end{figure}

\begin{prop}
Let $l_1,l_2\in\Rx_1$ linearly independent and let $w=\pm (l_1^2\otimes l_2^2-l_1l_2\otimes l_1l_2)\in W$. There is a full-dimensional \sa\ set of forms $f\in\Sigma_{3,4}$ such that the face of $\gram(f)$ in direction $w$ has dimension 2.
\end{prop}
\begin{proof}
We already know this is true for $l_1=x,l_2=y$, $w=R_1$ by \cref{lem:ternary_quartics_direction_R1}. Applying a linear \cc\ shows the claim.
\end{proof}

\begin{rem}
\label{rem:normal_cones_quartic}
Let $f\in\Sigma_{3,4}$ smooth with a 2-dimensional rank 5 face $F$ on $\gram(f)$ with corresponding subspace $U$, i.e. $U=\im\theta$ for any relative interior point $\theta\in F$. Let $\spn(q)=U^\perp$. Then $\rk(q)=2$ and $q\in\R[l_1,l_2]_2$ for some $l_1,l_2\in\Rx_1$. We have seen that the normal cone at $F$ is then given by $\pm (l_1^2\otimes l_2^2-l_1l_2\otimes l_1l_2)$.
\end{rem}

\begin{thm}
\label{thm:ternary_quartics_3-dim-face-and-nc}
Let $w=\pm (l_1^2\otimes l_2^2-l_1l_2\otimes l_1l_2)$. The face of $\Sigma_\mu K$ in direction $w$ has dimension 3 and a 1-dimensional normal cone spanned by $w$.
\end{thm}
\begin{proof}
We may assume $w=R_1$ after a \cc\ (or $w=-R_1$ but the argument stays the same). Consider the vector space
\[
W_1:=\spn(z^2\otimes l^2 -(zl)\otimes (zl)\colon l\in \Rx_1)
=\sum \ker(\mu|_{\sy{U_q}^2}\colon\sy{U_q}^2\to\Rx_4)
\]
where the sum is over all $q\in\R[x,y]_2$ with $\disc(q)<0$.

We claim that $W_1=\spn(R_2,R_3,R_6)$. All three $R_2,R_3,R_6$ have the form of the generators of $W_1$, and are therefore contained in $W_1$. On the other hand $R_1,R_4,R_5$ are in the orthogonal complement and thus the equality follows. 

One easily checks that removing a set of measure zero from $q\in\R[x,y]_2$, $\disc(q)<0$ does not alter the dimension, in fact any two generic kernels suffice already.

\cref{thm:dimension_face_fb} shows the claim about the dimension and \cref{rem:normal_cones_quartic}, \cref{thm:normal_cone_fb} show the statement about the normal cone.
\end{proof}

\begin{rem}
\cref{fig:ternary_quartic} shows the 3-dimensional face $(\Sigma_\mu K)^{R_1}$. For a full-dimensional subset of $\Sigma_{3,4}$ the face of $\gram(f)$ in direction $R_1$ is a 2-dimensional face whose affine hull (up to translation) is contained in the 3-dimensional space $W_1$. These 2-dimensional faces are rotated differently in this space which can be seen on the right in \cref{fig:ternary_quartic}. Taking the Minkowski sum of finitely many or constructing the fiber body then results in a 3-dimensional convex set.
\end{rem}

\begin{cor}
For every $0\neq w\in W$ that does not have the form in \cref{thm:ternary_quartics_3-dim-face-and-nc}, the face $(\Sigma_\mu K)^w$ is an extreme point.
\end{cor}
\begin{proof}
By \cref{rem:normal_cones_quartic} the face $\gram(f)^w$ is an extreme point for generic $f\in\Sigma_{3,4}$ as it is not a 2-dimensional face and there are no other positive dimensional faces. Therefore, $(\Sigma_\mu K)^w$ is also an extreme point.
\end{proof}

What is left to show for \cref{thm:ternary_quartics_fb} is the statement about the dimension of the normal cones.
For this we first determine the projection of rank 1 tensors to $W$.

\begin{lem}
\label{lem:ternary_quartics_coords_projection}
Let $q=a_{200}x^2+a_{020}y^2+a_{002}z^2+a_{110}xy+a_{101}xz+a_{011}yz\in\Rx_2$ and let $\theta=q\otimes q$. Denote by $Q$ the symmetric $3\times 3$-matrix associated to the quadratic form $q$. Then
\[
Q^{\adj}=\begin{pmatrix}
a_{200} & \frac{1}{2} a_{110} & \frac{1}{2}a_{101} \\
\frac{1}{2}a_{110} & a_{020} & \frac{1}{2}a_{011} \\
\frac{1}{2}a_{101} & \frac{1}{2}a_{011} & a_{002}
\end{pmatrix}^{\adj}
=\begin{pmatrix}
\bil{\theta}{R_3}& -\bil{\theta}{R_6}& -\bil{\theta}{R_5}\\
-\bil{\theta}{R_6}& \bil{\theta}{R_2}& -\bil{\theta}{R_4}\\
-\bil{\theta}{R_5}& -\bil{\theta}{R_4}& \bil{\theta}{R_1}
\end{pmatrix}
\]
\end{lem}
\begin{proof}
This is easily verified by hand.
\end{proof}

\begin{dfn}
For any $w\in W$, $w=\sum_{i=1}^6 \frac{\lambda_i}{\bil{R_i}{R_i}} R_i$, we define the corresponding quadratic form $Q(w)$ via the matrix
\[
Q(w):=\begin{pmatrix}
\lambda_3 & -\lambda_6 & -\lambda_5\\
-\lambda_6 & \lambda_2 & -\lambda_4\\
-\lambda_5 & -\lambda_4 & \lambda_1
\end{pmatrix}.
\]
\end{dfn}

To show that every face of $\Sigma_\mu K$ has a 1-dimensional normal cone we will consider several cases depending on the rank and the eigenvalues of $Q(w)$.

\begin{prop}[$\det Q(w)>0$]
\label{prop:positive_det_nc}
Let $w\in W$ such that $\det Q(w)>0$. Then the face of $\Sigma_\mu K$ in direction $w$ has a 1-dimensional normal cone.
\end{prop}
\begin{proof}
Let $A:=Q(w)^{\adj}$ and let $q\in\Rx_2$ be the associated quadratic form of $A$. Then $A^{\adj}=\det(Q(w))Q(w)$ and $\pr_W(q\otimes q)=\det(Q(w))w$ by \cref{lem:ternary_quartics_coords_projection}.
Hence, for every $f\in \interior\Sigma (\spn(q)^\perp)^2$ the face of $\gram(f)$ in direction $w$ is an extreme point with a 1-dimensional normal cone. Moreover, $U^2=\Rx_4$, hence \cref{thm:normal_cone_fb} implies the statement.
\end{proof}

\begin{rem}
This case is particularly easy because we can actually find a full-dimensional set in $\Sigma_{3,4}$ such that on every \gs\ the face in direction $w$ has a 1-dimensional normal cone. This wont be the case for other $w$.
\end{rem}

\begin{example}
\label{ex:1-dim-vanishing}
Let $U=\R[x,y]_2\oplus\spn(z^2)$, then $U^\perp=\spn(xz,yz)$. Let $f\in\interior\Sigma U^2$ and $q_1=2xz+2yz,\,q_2=2xz+4yz\in U^\perp$. The face of $\gram(f)$ in direction $\theta=q_1\otimes q_1+q_2\otimes q_2$ is a 1-dimensional face with image $U$. We calculate
\[
Q(\pr_W(\theta))=q_1^{\adj}+q_2^{\adj}=
\begin{pmatrix}
-1 & 1 & 0\\
1 & -1 & 0\\
0 & 0 & 0
\end{pmatrix}+
\begin{pmatrix}
-4 & 2 & 0\\1
2 & -1 & 0\\
0 & 0 & 0
\end{pmatrix}
=
\begin{pmatrix}
-5 & 3 & 0\\
3 & -2 & 0\\
0 & 0 & 0
\end{pmatrix}.
\]
This matrix has rank 2 ($\neq 1$) and therefore the face on $\Sigma_\mu K$ in direction $\pr_W(\theta)$ is an extreme point.
\end{example}

\begin{rem}
\label{rem:ternary_quartics_summary}
Let $w\in W$. If $\rk(Q(w))=1$ then the face in direction $w$ on $\Sigma_\mu K$ has dimension 3 and a 1-dimensional normal cone.
If $\det Q(w)$ is positive the face is an extreme point with a 1-dimensional normal cone.

Let $f\in\Sigma_{3,4}$ smooth. Let $F\subset\gram(f)$ be any rank 5 extreme point with corresponding subspace $U$ and orthogonal complement $\spn(q)=U^\perp$. Since $F$ is an extreme point $\rk(q)=3$. Let $Q$ be the associated symmetric matrix and let $A$ be its adjoint. Then $\det(A)\neq 0$ since $\rk(q)=3$ and $\det(A)=\det(Q)^2>0$. Hence the normal cone of $F$ is spanned by $w=\frac{4}{3} A_{33} R_1 +\frac{4}{3} A_{22} R_2 +\frac{4}{3} A_{11} R_3 -A_{23} R_4-A_{13}R_5-A_{12}R_6$.

Note however that for $w\in W$ with $\det Q(w)>0$, the face in direction $w$ on $\gram(f)$ is not necessarily a rank 5 point. This only follows for all $f\in\interior \Sigma U^2$ which is a full-dimensional \sa\ set but not all of $\Sigma_{3,4}$.
\end{rem}

Now we consider the other cases $\det Q(w)<0$ and $\rk Q(w)=2$. The difference is that we cannot find \gsa\ on which the faces in these directions have a 1-dimensional normal cone. Here, the faces always have a normal cone of dimension at least 3. However, we will see that nonetheless the normal cone of the fiber body in these directions has dimension 1.

\begin{prop}[$\det Q(w)<0$]
\label{prop:ternary_quartics_rk_2_completion}
Let $0\neq w\in W$ with $\det Q(w)<0$. Then there exists a full-dimensional \sa\ subset of $\Sigma_{3,4}$ such that for every $f$ in this subset the face of $\gram(f)$ in direction $w$ is a rank 4 extreme point.
\end{prop}
\begin{proof}
First, by diagonalizing it is easy to see that $Q(w)$ can be written as $Q(w)=Q_1+Q_2$ with $\det Q_i>0$ ($i=1,2$). Perturbing $Q_1$ we have $Q(w)=Q_1'+(Q(w)-Q_1')$ and still both matrices have positive determinant. Hence, there exists a full-dimensional subset $\oo$ of the set of symmetric $3\times 3$-matrices such that $Q(w)=Q_1+Q_2$ with $Q_1\in\oo$ and $\det Q_i>0$ ($i=1,2$).

For any such decomposition $Q(w)=Q(w_1)+Q(w_2)$ we have 
\[
\pr_W\left(\frac{1}{\det(Q(w_1))} q_1\otimes q_1+\frac{1}{\det(Q(w_2))} q_2\otimes q_2\right)=w_1+w_2=w
\]
where $q_i$ is the quadratic form associated to $Q(w_i)^{\adj}$ ($i=1,2$) as in \cref{prop:positive_det_nc}. Let $U=\spn(q_1,q_2)^\perp$, then for any $f\in\interior\Sigma U^2$ the face of $\gram(f)$ in direction $w$ is the rank 4 face corresponding to the subspace $U$. Since $\rk(q_i)=3$ it follows that $\dim U^2=10$ and this rank 4 face is an extreme point: Indeed, let $\spn(q)=\spn(q_1)^\perp\subset \spn(q_1,q_2)$. Then $(U\oplus\spn(q))^2=\Rx_4$ by \cref{prop:ternary_quartics_rk_5_faces}, i.e. the map $\sy{(U\oplus\spn(q))}\stackrel{\mu}{\to} \Rx_4$ is injective. Hence, $\sy{U}\to \Rx_4$ is also injective. We write $U_{q_1}:=U$ (or $U_{q_2}$). 
Similar to \cref{lem:ternary_quartics_direction_R1} we consider the set
\[
X=\bigcup_{q\in\oo} \interior\Sigma U_q^2
\]
We claim that $\dim X=15$. As in the proof of \cref{lem:ternary_quartics_direction_R1} every $f$ is contained in $\interior\Sigma U_q^2$ for at most one $q\in\oo$. Hence, for a generic $U=U_q$ there exists at least a 1-dimensional family of $q'\in\oo$ with $U=U_{q'}$. We need to show that this family has dimension exactly 1.

We need to check when two tensors of the form $\frac{1}{\det(Q(w_1))} q_1\otimes q_1+\frac{1}{\det(Q(w_2))} q_2\otimes q_2$ as above have the same image.
The image is given by 
\[
\spn(q_1,q_2)=\spn(Q(w_1)^{\adj},(Q(w)-Q(w_1))^{\adj}).
\]
We study two cases separately, depending on the number of positive eigenvalues of $Q(w)$.

(i): Assume $Q(w)$ has 2 positive eigenvalues $d_1,d_2$ and a negative eigenvalue $d_3$. After diagonalizing we can write
\[
\begin{pmatrix}
d_1 & 0 & 0\\
0 & d_2 & 0\\
0 & 0 & d_3
\end{pmatrix}
=
\begin{pmatrix}
\frac{d_1}{2} & 0 & 0\\
0 & 2d_2 & 0\\
0 & 0 & -d_3
\end{pmatrix}
+
\begin{pmatrix}
\frac{d_1}{2} & 0 & 0\\
0 & -d_2 & 0\\
0 & 0 & 2d_3
\end{pmatrix}
\]
which gives a decomposition $Q(w)=Q_1+Q_2$ with $\det Q_i>0$ ($i=1,2$) and $Q_1$ positive definite.
Since $Q(w_1)$ is positive definite there exists a change of coordinates (as quadratic forms) such that the image is given by $\spn\left(I,(D-I)^{-1}\right)$
where $I$ is the identity matrix and $D$ is some diagonal matrix. This space consists entirely of diagonal matrices. Therefore, if $\spn(q_1',q_2')=\spn(q_1,q_2)$ after the \cc\ the space is also spanned by two diagonal matrices and is given by $\spn\left(D_1,(D-D_1)^{-1}\right)$ with $D_1$ diagonal. We now parametrize the matrices $D,D_1$ and check in \texttt{Sage} \cite{sagemath} that for generic $D$, there is a 1-dimensional family of such diagonal matrices $D_1$.
This shows that the dimension of $X$ is indeed 15.

Thus there is a full-dimensional subset of $\Sigma_{3,4}$ for which the face in direction $w$ on $\gram(f)$ is a rank 4 extreme point, namely $X$.

(ii): Assume that $Q(w)$ has 3 negative eigenvalues $d_1,d_2,d_3$. We then get a decomposition
\[
\begin{pmatrix}
d_1 & 0 & 0\\
0 & d_2 & 0\\
0 & 0 & d_3
\end{pmatrix}
=
\begin{pmatrix}
2d_1 & 0 & 0\\
0 & -d_2 & 0\\
0 & 0 & \frac{d_3}{2}
\end{pmatrix}
+
\begin{pmatrix}
-d_1 & 0 & 0\\
0 & 2d_2 & 0\\
0 & 0 & \frac{d_3}{2}
\end{pmatrix}
\]
into two matrices with $\det Q_i>0$ ($i=1,2$) but both not positive definite. We have 
\[
\spn(Q(w_1)^{\adj},(Q(w)-Q(w_1))^{\adj})=\spn(Q(w_1)^{\adj},(-Q(w)+Q(w_1))^{\adj}).
\]
Now $-Q(w)$ is positive definite and we can find a \cc\ that makes $-Q(w)$ the identity and diagonalizes $Q(w_1)$. Hence, we have to consider the space $\spn(D^{-1},(I+D)^{-1})$. We now use the same argument as in case (i) and are finished.
\end{proof}

\begin{rem}[$\rk Q(w)=2$]
Now let $\rk Q(w)=2$. Again we separately consider the cases depending on the signs of the eigenvalues. If at least one eigenvalue is positive we get a decomposition where one matrix is positive definite. Indeed, if $d_1,d_2>0$
\[
\begin{pmatrix}
d_1 & 0 & 0\\
0 & d_2 & 0\\
0 & 0 & 0
\end{pmatrix}
=
\begin{pmatrix}
\frac{d_1}{2} & 0 & 0\\
0 & 2d_2 & 0\\
0 & 0 & 1
\end{pmatrix}
+
\begin{pmatrix}
\frac{d_1}{2} & 0 & 0\\
0 & -d_2 & 0\\
0 & 0 & -1
\end{pmatrix}
\]
and if $d_1>0,\, d_2<0$ we have 
\[
\begin{pmatrix}
d_1 & 0 & 0\\
0 & d_2 & 0\\
0 & 0 & 0
\end{pmatrix}
=
\begin{pmatrix}
\frac{d_1}{2} & 0 & 0\\
0 & -d_2 & 0\\
0 & 0 & 1
\end{pmatrix}
+
\begin{pmatrix}
\frac{d_1}{2} & 0 & 0\\
0 & 2d_2 & 0\\
0 & 0 & -1
\end{pmatrix}.
\]
Again we check with \texttt{Sage} that the fibers have the correct dimension. We are left with the case where $Q(w)$ is negative semidefinite. As in (ii) we make $-Q(w)$ a diagonal matrix with two 1's on the diagonal. In this case we cannot diagonalize any other matrix, so we fully parametrize the two symmetric matrices and solve the problem in \texttt{Sage}. The fibers still have dimension 1.
\end{rem}

\begin{rem}
We briefly explain how to check the fiber dimension in \texttt{Sage}. We do this in case (i) only, the rest works analogously. Consider the set 
\[
\left\lbrace(D,D_1)\in (\R^3)^2\colon \spn\left(I,(D-I)^{\adj}\right)=\spn\left(D_1^{\adj},(D-D_1)^{\adj}\right)\right\rbrace.
\]
Its Zariski-closure is defined by the ideal of $3\times 3$ minors of the matrix containing $I,(D-I)^{\adj},D_1^{\adj},(D-D_1)^{\adj}$ after writing them into a $4\times 9$ matrix. The dimension of this variety is 4 and the projection to the second factor is dense in $\R^3$. Hence, a generic fiber has dimension 1, i.e. for a generic $D$ there exists a 1-dimensional set of $D_1$ such that the two subspaces are equal.
\end{rem}

We can now easily see that the dimension of the normal cones of faces of $\Sigma_\mu K$ is at most 3.

\begin{cor}
Every boundary point of $\Sigma_\mu K$ has a normal cone of dimension at most 3.
\end{cor}
\begin{proof}
Let $F\subset\Sigma_\mu K$ be an exposed face with normal cone $N$ and let $w\in N$. If $\rk Q(w)=1$ or $\det Q(w)>0$, $F$ has a 1-dimensional normal cone by \cref{rem:ternary_quartics_summary}. If $\det Q(w)<0$ it follows from \cref{prop:ternary_quartics_rk_2_completion} and \cref{thm:normal_cone_fb} that the normal cone has dimension at most 3. Thus we may now assume that every $w\in N$ satisfies $\rk Q(w)=2$.
By \cite{ci2015} the Fano variety of the determinant of symmetric $3\times 3$ matrices only contains 3-dimensional vector spaces at most and no 4-dimensional. Therefore, the dimension of $N$ is at most 3.
\end{proof}

\begin{prop}[$\det Q(w)<0,\, \rk Q(w)=2$]
Let $0\neq w\in W$ with $\det Q(w)<0$ or $\rk Q(w)=2$. The normal cone of $(\Sigma_\mu K)^w$ has dimension 1.
\end{prop}
\begin{proof}
Let $w'$ be a perturbation of $w$ which is not a scalar multiple. As in the proof of \cref{prop:ternary_quartics_rk_2_completion} we consider the full-dimensional \sa\ subsets of $\Sigma_{3,4}$
\[
X_w=\bigcup_{q\in\oo_w} \interior\Sigma U_q^2,\quad X_{w'}=\bigcup_{q\in\oo_{w'}} \interior\Sigma U_q^2
\]
with $\oo_w,\oo_{w'}\subset\Rx_2$ with non-zero measure as in \cref{prop:ternary_quartics_rk_2_completion}. 
Since $w'$ is a perturbation of $w$, the two full-dimensional sets $X_w, X_{w'}$ intersect again in a full-dimensional \sa\ set $X$. Let $f\in X$ be generic and $f\in\Sigma U_{q_1}^2\cap \Sigma U_{q_2}^2$ with $q_1\in\oo_w,\, q_2\in\oo_{w'}$. We claim that $U_{q_1}\neq U_{q_2}$. 
After showing this, we see that for every $f\in X$ the faces in directions $w$ and $w'$ differ, hence the faces $(\Sigma_\mu K)^w$ and $(\Sigma_\mu K)^{w'}$ differ and thus the normal cone of $(\Sigma_\mu K)^w$ has dimension 1.
Assume this is false. To simplify the notation we will not distinguish between a quadratic form $q\in\Rx_2$ and its associated symmetric matrix in this proof, and also write $w,w'$ instead of $Q(w),Q(w')$.

(i): Assume we can find a decomposition $Q(w)=Q_1+Q_2$ with $Q_1$ positive definite and $\det Q_2>0$.
Then there is a full-dimensional subset $\oo_w'\subset\oo_w$ such that for every $Q_1\in\oo_w'$ there is $Q_2\in\oo_{w'}$ with
\[
\spn(Q_1^{\adj},(w-Q_1)^{\adj})=\spn(Q_2^{\adj},(w'-Q_2)^{\adj}).
\]
Since $Q_1$ is positive definite there exists $S\in\gl_3(\R)$ such that $SQ_1S^t=I$ and $SwS^t=D$ diagonal. It follows that also $SQ_2S^t$ and $Sw'S^t$ are diagonal. 
Hence, for every $Q\in\oo_w'$ the three quadratic forms $Q,w,w'$ are simultaneously diagonalizable. By \cite[Theorem 10]{jl2016} this is equivalent to $SwS^t$ and $Sw'S^t$ commuting, i.e.
\[
(SwS^t)(Sw'S^t)=(Sw'S^t)(SwS^t) \iff \underbrace{w'^{-1}w}_{=:A}=(S^tS)\underbrace{(ww'^{-1})}_{=:B}(S^tS)^{-1}.
\]
Since $\oo_w'$ is full-dimensional, we find a full-dimensional subset $\uu\subset\gl_3(\R)\cap\Sym_3^+$ such that for every $G\in\uu$ we have $A=GBG^{-1}$. Let $G\in\uu$ be an interior point of $\uu$ and write $TGT^t=I$, $T\in\gl_3(\R)$. Then we may assume that $\uu':=T\uu T^t$ is a neighbourhood of $I$ satisfying
\[
(TPT^t)\underbrace{((T^t)^{-1}BT^t)}_{=:B'}(TPT^t)^{-1}=TPBP^{-1}T^{-1}=\underbrace{TAT^{-1}}_{=:A'}
\]
for $P\in\uu$. Since $I\in\uu'$ it follows that $B'=A'$. Let $D=\diag(d_1,d_2,d_3)$ be a diagonal matrix in $\uu'$ then $DB'D^{-1}=A'=B'$. Since $\uu'$ is a neighbourhood of $I$ it follows that $B'$ is a diagonal matrix. Now, using the matrices 
\[
\begin{pmatrix}
1 & \epsilon & 0\\
\epsilon & 1 & 0\\
0 & 0 & 1
\end{pmatrix},
\begin{pmatrix}
1 & 0 & 0\\
0 & 1 & \epsilon\\
0 & \epsilon & 1
\end{pmatrix}
\]
and comparing entries shows that $B'=\lambda I$ for some $\lambda\in\R$. Thus, the same holds for $A$ and $B$. Especially, $w'=\lambda w$, a contradiction.

(ii): Assume $Q(w)$ is negative definite. Let $S\in\gl_3(\R)$ such that $-SwS^t=I$ and $SQ_1S^t$ is diagonal. Then also $Sw'S^t$ is diagonal. Hence, $w,w'$ and any $Q_1\in\oo_w'$ are simultaneously diagonalizable. By \cite[Theorem 10]{jl2016} this means that
\[
(SQ_1S^t)(Sw'S^t)=(Sw'S^t)(SQ_1S^t) \iff (Sw'S^t)=(SQ_1S^t)^{-1}(Sw'S^t)(SQ_1S^t)
\]
(note that even though $S$ depends on $Q_1$, we may choose any $S$ that makes $-w$ the identity matrix in this equation by \cite[Theorem 10]{jl2016}). However, this means that $Sw'S^t$ is stabilized by a full-dimensional subset of $\gl_3(\R)\cap\Sym_3$ and $\det (Sw'S^t)\neq 0$ which is only true for $\lambda I$, $\lambda\in\R$. Hence, $Sw'S^t=\lambda I$ and by assumption $SwS^t=I$, i.e. $w=S^{-1}(S^t)^{-1}$ and $w'=\lambda w$.

(iii): $\rk Q(w)=2$ and $Q(w)$ is negative semidefinite. If one of $w,w'$ is invertible we are in cases (i) or (ii). Hence, we may assume that both have rank 2. Let $\norm{\cdot}$ be any matrix norm with $\norm{T}=\norm{T^t}$ for all real $3\times 3$ matrices, e.g. the one induced by the trace bilinear form. We apply \cite[Lemma 5]{jl2016} which states that we can find $S\in\gl_3(\R),\, \norm{S}=1$, such that
\[
SwS^t=\begin{pmatrix}
A & 0\\
0 & 0
\end{pmatrix},\quad 
Sw'S^t=\begin{pmatrix}
B & 0\\
0 & s
\end{pmatrix}
\text{ or }
Sw'S^t=\begin{pmatrix}
B & s\\
s & 0
\end{pmatrix}
\]
where $A$ is a diagonal matrix of rank 2, $B$ is a $2\times 2$ symmetric matrix and $s\in\R$. We have
\[
\norm{SwS^t-Sw'S^t}=\norm{S(w-w')S^t}\le \norm{w-w'}\norm{S}\norm{S^t}=\norm{w-w'}.
\]
Hence $B$ is a perturbation of $A$ and thus has rank 2. Since $w'$ has rank 2 we see that $s=0$. But now both matrices $SwS^t$ and $Sw'S^t$ are block matrices only with an upper left $2\times 2$ block not being zero. Hence, these two can be simultaneously diagonalized as they are basically $2\times 2$ matrices. Since $w$ is negative semidefinite we may now assume that $S\in\gl_3(\R)$ satisfies
\[
SwS^t=\begin{pmatrix}
-1 & 0 & 0\\
0 & -1 & 0\\
0 & 0 & 0
\end{pmatrix}=:A',\quad
Sw'S^t=\begin{pmatrix}
a & 0 & 0\\
0 & b & 0\\
0 & 0 & 0
\end{pmatrix}=:B'.
\]
Therefore, for every $Q_1\in\oo_w'$ there exists $Q_2\in\oo_{w'}$ such that
\[
\spn(Q_1^{\adj},(A'-Q_1)^{\adj})=\spn(Q_2^{\adj},(B'-Q_2)^{\adj})
\]
where $B'$ is also negative semidefinite as it is a perturbation of $A'$ of rank 2. We write
\[
Q_1=\begin{pmatrix}
c_{1} & c_{4} & c_{5} \\
c_{4} & c_{2} & c_{6} \\
c_{5} & c_{6} & c_{3}
\end{pmatrix},\quad 
Q_2=\begin{pmatrix}
a_{1} & a_{4} & a_{5} \\
a_{4} & a_{2} & a_{6} \\
a_{5} & a_{6} & a_{3}
\end{pmatrix}
\]
and calculate the difference of the two generators of both spaces. This gives
\[
Q_1^{\adj}-(A'-Q_1)^{\adj}=\scalemath{0.8}{\begin{pmatrix}
\minus c_{3} & 0 & c_{5} \\
0 & \minus c_{3} & c_{6} \\
c_{5} & c_{6} & \minus c_{1} \minus  c_{2} \minus  1
\end{pmatrix}},\,\,
Q_2^{\adj}-(B'-Q_2)^{\adj}=\scalemath{0.8}{\begin{pmatrix}
b a_{3} & 0 & \minus b a_{5} \\
0 & a a_{3} & \minus a a_{6} \\
\minus b a_{5} & \minus a a_{6} & \minus a b + b a_{1} + a a_{2}
\end{pmatrix}}.
\]
If $Q_1$ is chosen generically then the $(1,2)$-entry of $Q_1^{\adj}$ and $c_3$ are both non-zero. The first assumption implies the equality 
\[
Q_1^{\adj}-(A'-Q_1)^{\adj}=\lambda (Q_2^{\adj}-(B'-Q_2)^{\adj})
\]
for some $\lambda\in\R$. From this we get $0\neq -c_3=\lambda ba_3=\lambda aa_3$ and thus $a=b$ which shows that $B'$ is a positive scalar multiple of $A$.
\end{proof}

This finishes the proof of \cref{thm:ternary_quartics_fb}.

\begin{rem}
We do not know if all extreme points on the boundary of the 3-dimensional faces are exposed faces. As in the case of binary sextics no other faces can be non-exposed since all other boundary points are extreme points.
\end{rem}

\begin{rem}
We discuss \cite[Table 2]{psv2011} which shows some statistics when optimizing linear functionals over \gsa\ of ternary quartics and also contains the algebraic degree depending on the rank of the optimal solution.

Firstly, the probabilities found there compare well to our findings about the normal cones. If $w\in W$ is chosen randomly, then half the time $\det Q(w)<0$. If $\det Q(w)>0$ the face in direction $w$ is a rank 5 extreme point only for a full-dimensional \sa\ subset of $\Sigma_{3,4}$. Therefore, the probability of finding a rank 5 point has to be rather low. Moreover, minimizing $w\in W$ with $\det Q(w)<0$ over the \gs\ of any smooth psd ternary quartic yields a rank 3 or rank 4 point.

Secondly, we explain the algebraic degree of the optimal solution for rank 5 extreme points which was found to be 1.
\end{rem}

\begin{prop}
\label{prop:algebraic_degree}
Let $f\in\Sigma_{3,4}$ rational, smooth and $0\neq w\in W$ rational such that the face of $\gram(f)$ in direction $w$ is a rank 5 extreme point $\theta$. Then $\theta$ can be defined over $\Q$ and the sos representation of $f$ corresponding to $\theta$ as well.
\end{prop}
\begin{proof}
Since $w$ spans the normal cone at $\theta$ we know $\det Q(w)>0$. We may therefore find $q\in\Rx_2$ such that $\pr_W(q\otimes q)=\lambda w$ for some $\lambda\in \R_{>0}$. Then $f\in\interior\Sigma (\spn(q)^\perp)^2$ and the image of $\theta$ is given by $\spn(q)^\perp$. We see from the proof of \cref{prop:positive_det_nc} that $q$ is the adjoint of $Q(w)$ and therefore rational if $Q(w)$ is. Thus $\spn(q)^\perp$ has a $\Q$-basis. The Gram tensor $\theta$ gives rise to a sos representation $f=\sum_{i=1}^5 f_i^2$ with $f_1,\dots,f_5$ a basis of $\spn(q)^\perp$. It now follows from \cite[Lemma 4.4]{cs2020} that this representation of $f$ is defined over $\Q$. Thus $\theta=\sum_{i=1}^5 f_i\otimes f_i$ is also rational.
\end{proof}

\section{Questions}

We finish by formulating several questions arising from our results.

\begin{rem}
In general the fiber body is not a \sa\ set. However, after seeing that in the case of binary sextics the inequalities describing the full-dimensional normal cone are particularly easy (\cref{thm:binary_sextics_3_dim_nc}) and especially showing it is a \sa\ set, one may ask if at least in this special case the fiber body itself is \sa\ as well.

The same could of course also be asked about the fiber body in the case of ternary quartics or about the 3-dimensional faces in that case. However, this seems more unlikely as for example the 3-dimensional faces already make a jump in dimension as they come from 2-dimensional faces on \gsa.
\end{rem}

\begin{rem}
We may also continue studying the fiber body of \gsa\ of binary forms of higher degree and ask if there is still an extreme point with a full-dimensional normal cone. It is not hard to see that the higher dimensional analogous directions to $R_1,R_3$ also have a special interaction with the rank 2 Gram tensor corresponding to the sos representation where all zeros of the binary form are grouped by the sign of the imaginary part.

The same interpretation as in the case of binary sextics would then apply and every direction/linear functional contained in this normal cone minimizes on a rank 2 Gram tensor for almost all binary forms of fixed degree. 

Note that in higher degrees the algebraic boundary of \gsa\ is a hypersurface that is ruled by positive dimensional subspaces. Thus a general point on the boundary of the fiber body will also not be an extreme point any more.
\end{rem}

\begin{rem}
Are there other interesting families of spectrahedra where the fiber body can be understood well? Is it \sa\ in that case? This might be particularly easy if the completion approach described in \cref{rem:completion} which we constantly use behaves well.

What about families of spectrahedra whose algebraic boundary is a quartic symmetroid for example (see \cite{orsv2015})? In this case, do fiber bodies of 'nice' families still have distinguished extreme points on their boundary? Fiber bodies of 'general' families might well lose all special points, i.e. their boundary consists only of extreme points with 1-dimensional normal cones. Are they \sa?
\end{rem}

\begin{rem}
Another question about spectrahedra in general might be as follows: Assume we are given the facial structure of the fiber body for a linear map. What can we conclude about generic fibers? Of course this will be very hard in general. If however a general point on the boundary is an extreme point and positive dimensional faces either have low dimension or are close to having codimension 1, one might be able to understand the fibers from the fiber body.
\end{rem}

{\bf Acknowledgements}. 
I would like to thank Bernd Sturmfels for suggesting this problem during his visit at Universit\"at Konstanz. Most grateful I am to Chiara Meroni for numerous discussions on the topic and answering countless questions about fiber bodies.

\bibliographystyle{plain}

\end{document}